\pdfoutput=1
\documentclass[11pt,a4paper]{article}
\usepackage[margin=2.5cm]{geometry}
\usepackage{amsmath,amssymb,amsthm}
\usepackage{mathtools}
\usepackage{xcolor}
\usepackage{enumitem}
\usepackage{cite}
\usepackage{tikz}
\usetikzlibrary{positioning,arrows.meta,calc}
\usepackage[colorlinks=true,linkcolor=blue!60!black,citecolor=blue!60!black,urlcolor=blue!60!black]{hyperref}

\theoremstyle{plain}
\newtheorem{thrm}{Theorem}[section]
\newtheorem{lmm}[thrm]{Lemma}
\newtheorem{crllr}[thrm]{Corollary}
\newtheorem{prpstn}[thrm]{Proposition}
\theoremstyle{definition}
\newtheorem{dfntn}[thrm]{Definition}
\newtheorem{rmrk}[thrm]{Remark}

\newcommand{\Rref}[1]{\textup{\ref{#1}}}

\newcommand{\xopt}{x^*}
\newcommand{\uopt}{u^*}
\newcommand{\lopt}{\lambda}
\newcommand{\popt}{p}
\newcommand{\yopt}{y^*}
\newcommand{\ropt}{r}

\numberwithin{equation}{section}
\allowdisplaybreaks

\newcommand{\dom}[1]{D(#1)}
\newcommand{\domain}[1]{\operatorname{dom}(#1)}
\newcommand{\ran}[1]{R(#1)}
\newcommand{\diag}{\operatorname{diag}}

\newcommand{\calA}{\mathcal A}
\newcommand{\calB}{\mathcal B}

\newcommand{\calP}{\mathcal P}
\newcommand{\calH}{\mathcal H}
\newcommand{\calU}{\mathcal U}
\newcommand{\Uad}{\calU_\mathrm{ad}}
\newcommand{\prox}{\operatorname{prox}}
\newcommand{\calX}{\mathcal X}

\newcommand{\Xd}{\calX^{\mathrm d}}
\newcommand{\Ed}{E^{\mathrm d}}
\newcommand{\Ad}{\calA^{\mathrm d}}
\newcommand{\tXd}{\widetilde{\calX}^{\mathrm d}}
\newcommand{\tEd}{\widetilde{E}^{\mathrm d}}
\newcommand{\tAd}{\widetilde{\calA}^{\mathrm d}}

\newcommand{\stack}[2]{\left[ \begin{smallmatrix} #1 \\ #2 \end{smallmatrix} \right]}

\newcommand{\N}{\mathbb{N}}
\newcommand{\R}{\mathbb{R}}

\begin{document}

\title{Dual companions and improved regularity of Lagrange multipliers in optimal control\thanks{The authors acknowledge funding by the Deutsche Forschungsgemeinschaft (DFG, German Research Foundation) -- Project-ID 531152215 -- CRC 1701.}}
\author{Till Preuster\thanks{Junior Professorship Numerical Mathematics, Faculty of Mathematics, Chemnitz University of Technology, Germany. E-mail: \texttt{\{till.preuster,manuel.schaller\}@math.tu-chemnitz.de}}
\and Manuel Schaller\footnotemark[2]}
\date{}

\maketitle

\begin{abstract}
We investigate optimal control problems whose state equation takes the form $\calA x-\calB u=b$, where $\calA:\calX\to\calP'$ is an isomorphism from a Banach state space $\calX$ onto the dual of a reflexive Banach space $\calP$. The adjoint equation involves the Banach space adjoint $\calA':\calP\to\calX'$, so that, in general, state and multiplier are of different regularity. For particular classes of partial differential equations, the typically weaker regularity of the multiplier is improved a posteriori by bootstrapping. In this work, we provide an abstract foundation for this methodology on the operator level. More precisely, we show that whenever the state cost extends to a differentiable function on a Hilbert space $\calH$ in which $\calX$ is densely embedded, as for $L^2$-tracking costs on Sobolev-type state spaces, the multiplier automatically lies in a subspace of $\calP$, the dual companion, on which $\calA'$ restricts to an isomorphism onto $\calH'$. This yields an optimality system with balanced regularity, where state and multiplier solve equations of the same type. In applications, the dual companion can be identified by integration by parts in time or space. We illustrate the framework by linear elliptic and parabolic equations, Cauchy problems governed by general $C_0$-semigroups, and a wave equation with accumulated-state observation.
\end{abstract}

\medskip


\smallskip


\section{Introduction}
\noindent Optimal control problems subject to partial differential equations are a well-established field with a rich literature, see, e.g., the monographs \cite{lions1971optimal,troltzsch2010optimal,hinze2009optimization} and, for dynamical systems governed by operator semigroups, \cite{bensoussan2007representation,lasiecka2000control}. A convenient way to treat such problems in a unified manner is to write the state equation as an operator equation, cf.\ \cite{schiela2013concise}, which leads to the abstract problem
\begin{equation}\label{eq:OCP_intro}
	\min_{(x,u)\in\calX\times\calU}\ J(x)+\tfrac{\alpha}{2}\|u\|_\calU^2+\theta(u)\quad\mathrm{s.t.}\quad\calA x-\calB u=b,
\end{equation}
with a state cost $J:\calX\to\R$, a regularization parameter $\alpha>0$, a convex and possibly nonsmooth control cost $\theta:\calU\to\R\cup\{+\infty\}$, a control operator $\calB:\calU\to\calP'$ and an isomorphism $\calA:\calX\to\calP'$, where the state space $\calX$ is a Banach space, $\calU$ is a Hilbert space and the data space $\calP'$ is the dual space of a reflexive Banach space $\calP$. Here, $\calA$ comprises the differential equation together with its initial or boundary conditions, and its bounded invertibility encodes the well-posedness of the state equation. This formulation separates the analysis of the state equation from the optimization. For elliptic problems in weak form, $\calA$ is an isomorphism by the Lax--Milgram lemma, and in strong form by elliptic regularity, see, e.g., \cite[Thm.~3.2.1.2]{grisvard1985elliptic}. For parabolic problems, such isomorphisms are well known, both in the variational setting and, via maximal regularity, for generators of analytic semigroups, see, e.g., \cite[Part~II, Ch.~1 and~2]{bensoussan2007representation} and \cite{schiela2013concise}, and they are at the core of space-time discretizations of evolution equations, see, e.g., \cite{schwab2009space,steinbach2015space}, including ultraweak formulations of transport and wave equations, whose test spaces are tailored to the adjoint operator, see, e.g., \cite{dahmen2012adaptive,brunken2019parametrized,henning2022ultraweak}.

In this framework, first-order optimality conditions follow from standard arguments of convex analysis and yield a Lagrange multiplier $\popt\in\calP$, which satisfies the adjoint equation $\calA'\popt=-DJ(\xopt)$ with the Banach space adjoint $\calA'$ of $\calA$. This leads to an imbalance in regularity between state and multiplier. Whereas the state equation is posed in $\calP'$ and solved by the state $\xopt\in\calX$, the adjoint equation is posed in the dual space $\calX'$ and, a priori, the multiplier is only an element of $\calP$. For evolution equations, this means that the state solves a forward problem in a strong or mild sense, whereas the adjoint equation only holds in a very weak sense, i.e., in the sense of the method of transposition, cf. \cite[Ch.~2, Sec.~6]{lions1972nonhomogeneous} and \cite[Part~II, Ch.~2, Sec.~2]{bensoussan2007representation}, in which the multiplier neither possesses initial or terminal values nor is continuous in time. This is in contrast to the way optimality systems are used in practice, where the adjoint state is interpreted as the solution of an evolution equation posed backward in time, see, e.g., \cite{lions1971optimal,troltzsch2010optimal,trelat2018steady,grune2020exponential}. This interpretation is the basis of forward-backward solution methods, of the derivation of regularity and structural properties of optimal controls via projection formulas, and of error estimates for discretizations, in which state and adjoint are approximated in spaces of comparable regularity.

It is well known that, in many situations, the multiplier is indeed more regular than guaranteed a priori. For elliptic problems, the adjoint state solves a boundary value problem whose right-hand side inherits the regularity of the cost functional, such that elliptic regularity theory lifts the adjoint state beyond the form domain. Via the projection formula, this regularity carries over to the optimal control, which is a standard ingredient in the error analysis of discretized optimal control problems, see, e.g., \cite{troltzsch2010optimal,hinze2009optimization,meyer2004superconvergence,hinze2005variational}, and in the analysis of the structure of sparse optimal controls, see, e.g., \cite{stadler2009elliptic,casas2012optimality}. For evolution equations, the adjoint equation can be interpreted as an evolution equation posed backward in time. For linear parabolic equations in a variational setting, \cite{schiela2013concise} presents a concise proof of the bounded invertibility of the corresponding operator and of its adjoint and shows that the adjoint state can be lifted a posteriori into the solution space of the state, provided that the cost functional is sufficiently regular. However, all of these results rely on the particular structure of the respective equation, such as coercivity and elliptic regularity for elliptic problems, the Gelfand triplet structure of the variational formulation for parabolic problems, or semigroup arguments for hyperbolic problems. Thus, each class of equations is treated separately, and a general mechanism, which explains when and why the multiplier gains regularity, seems to be missing.

\smallskip
\noindent\textbf{Contribution and outline.} In this paper, we propose an operator-theoretic approach to this question, which covers elliptic, parabolic and hyperbolic equations as well as ODEs and does not depend on the particular structure of the state equation. It only requires that $\calA$ is an isomorphism and that the state cost is of the form $J=f\circ E$ with a differentiable function $f$ on a reflexive Banach space $\calH$, e.g., a Hilbert space, and a dense embedding $E:\calX\to\calH$. After a motivating example with linear ODEs in Section~\ref{sec:ode_example}, we introduce this setting in Section~\ref{sec:abstract_ocp} and derive first-order optimality conditions. Our main insight is that the additional regularity of the multiplier is not a property of the differential equation, but of the cost functional. Precisely, if $DJ(\xopt)$ is more regular than a general element of $\calX'$, then so is the solution of the adjoint equation. To make this precise, we introduce in Section~\ref{sec:balancing} the dual companion $(\Xd,\Ed,\Ad)$ of $(\calX,E,\calA)$, which consists of an embedding $\Ed$ into $\calP$ and an isomorphism $\Ad$ onto the dual space $\calH'$ satisfying $\calA'\Ed=E'\Ad$, see Proposition~\ref{prop:dual_companion}.
The dual companion is unique up to isomorphism, can be identified in applications by an integration by parts formula together with the bounded invertibility of a candidate operator, see Proposition~\ref{prop:companion_unique}, and is symmetric in the sense that the dual companion of the dual companion is, up to representation, the original triple, see Corollary~\ref{cor:companion_duality}. As a consequence, the multiplier is of the form $\popt=\Ed\lopt$ with $\lopt\in\Xd$, and state and adjoint state solve equations of the same type, see Theorem~\ref{thm:adjoint_regularity_general} and Figure~\ref{fig:balancing}. As the control cost only enters via its proximal operator, our results hold for general convex and possibly nonsmooth control costs, which include control constraints and sparsity-promoting costs.
In Section~\ref{sec:examples_ocp}, we illustrate the generality of the approach by four classes of examples. For linear elliptic PDEs in weak form, no imbalance occurs, but the dual companion recovers the well-known additional regularity of the adjoint state and of the optimal control. For linear parabolic PDEs in a variational setting, we recover the a posteriori lifting of the adjoint state of \cite{schiela2013concise}. For abstract Cauchy problems governed by general $C_0$-semigroups, which include linear ODEs as well as hyperbolic PDEs, we first show that the operator mapping a state to its residual and its initial value is an isomorphism, which extends the corresponding result for analytic semigroups in \cite[Part~II, Ch.~1, Thm.~3.1]{bensoussan2007representation}, and then identify the dual companion as the backward Cauchy problem governed by the adjoint generator, without any assumption on analyticity or maximal regularity. Finally, a wave equation with a cost functional acting on the accumulated state shows that the dual companion is not necessarily a standard backward evolution equation, but is determined by the structure of the cost functional. Section~\ref{sec:conclusion} concludes the paper.

\section{A motivating example: optimal control of linear ODEs}\label{sec:ode_example}
	We start with a motivating example, which illustrates the imbalance between state and multiplier in the optimality system and how it can be resolved. The same phenomenon occurs for abstract Cauchy problems governed by general $C_0$-semigroups as well as for parabolic and hyperbolic PDEs, cf.\ Section~\ref{sec:examples_ocp}. To keep the presentation simple, we restrict ourselves to linear ODEs in this section.

	Let $T>0$, $A\in\R^{n\times n}$, $B\in\R^{n\times m}$, $x_0\in\R^n$, $x_d\in L^2(0,T;\R^n)$ and $\alpha>0$, let $\varphi:\R^n\to\R$ be convex and differentiable, and let $U_\mathrm{ad}\subset\R^m$ be nonempty, closed and convex. We consider
	\begin{equation}\label{eq:OCP_ODE}
		\begin{aligned}
			\min_{(x,u)}\ & \frac12\int_0^T \|x(t)-x_d(t)\|_{\R^n}^2\,\mathrm{d}t+\varphi(x(T))+\frac\alpha2\int_0^T\|u(t)\|_{\R^m}^2\,\mathrm{d}t\\
			\mathrm{s.t.}\ & \dot x(t)=Ax(t)+Bu(t),\quad x(0)=x_0,\quad u(t)\in U_\mathrm{ad} \ \text{ for a.e. } t\in(0,T).
		\end{aligned}
	\end{equation}
	We choose the state space $\calX\coloneqq H^1(0,T;\R^n)$ and the control space $\calU\coloneqq L^2(0,T;\R^m)$. Identifying $\calP\coloneqq L^2(0,T;\R^n)\times\R^n$ with its dual space $\calP'$ via its inner product, we write the dynamics as the operator equation $\calA x-\calB u=(0,x_0)$ with
	\begin{equation*}
		\calA:\calX\to\calP',\quad \calA x=\stack{\dot x-Ax}{x(0)},\qquad\qquad \calB:\calU\to\calP',\quad \calB u=\stack{Bu}{0}.
	\end{equation*}
	By the variation of constants formula, $\calA$ is an isomorphism with inverse
	\begin{equation*}
		\calA^{-1}\stack{g}{\xi}=e^{A\cdot}\xi+\int_0^\cdot e^{A(\cdot-s)}g(s)\,\mathrm{d}s,\qquad \stack{g}{\xi}\in\calP'.
	\end{equation*}
	Hence, with $J(x)\coloneqq\frac12\|x-x_d\|^2_{L^2(0,T;\R^n)}+\varphi(x(T))$ and $\Uad\coloneqq\{u\in\calU: u(t)\in U_\mathrm{ad}\text{ for a.e. } t\in(0,T)\}$, problem \eqref{eq:OCP_ODE} is of the form \eqref{eq:OCP_intro}, i.e.,
	\begin{equation*}
		\min_{(x,u)\in \calX\times\calU} J(x)+\frac\alpha2\|u\|^2_{\calU}+\iota_{\Uad}(u)\quad \mathrm{s.t.}\quad \calA x-\calB u=\stack{0}{x_0},
	\end{equation*}
	where the control cost $\theta=\iota_{\Uad}$ is the \emph{indicator function} of $\Uad$. Here, the indicator function $\iota_C$ of a set $C$ is given by $\iota_C(u)=0$ for $u\in C$ and $\iota_C(u)=+\infty$ otherwise.

	Problem \eqref{eq:OCP_ODE} has a unique solution $(\xopt,\uopt)$, and standard Lagrange multiplier arguments, see also Theorem~\ref{thm:existence_optimal_sol} below, yield a multiplier $\popt=(\lopt,\lopt_0)\in\calP$ with $\calA'\popt=-DJ(\xopt)$ in $\calX'$, i.e.,
	\begin{equation}\label{eq:ODE_adjoint_weak}
		\int_0^T \langle \lopt(t), \dot h(t)-Ah(t)\rangle_{\R^n}\,\mathrm{d}t+\langle\lopt_0,h(0)\rangle_{\R^n} = -\int_0^T\langle \xopt(t)-x_d(t),h(t)\rangle_{\R^n}\,\mathrm{d}t-\langle\nabla\varphi(\xopt(T)),h(T)\rangle_{\R^n}
	\end{equation}
	for all $h\in\calX$, and $\uopt=P_{U_\mathrm{ad}}(\frac1\alpha B^\top\lopt)$, where the orthogonal projection $P_{U_\mathrm{ad}}$ onto $U_\mathrm{ad}$ is applied pointwise. Thus, the state equation is posed in $\calP'$ and solved by $\xopt\in\calX$, whereas the adjoint equation \eqref{eq:ODE_adjoint_weak} is posed in the dual space $\calX'$ and, a priori, only provides $\lopt\in L^2(0,T;\R^n)$.

	However, it is well-known that the multiplier enjoys higher regularity. Indeed, choosing $h\in C_c^\infty(0,T;\R^n)$ in \eqref{eq:ODE_adjoint_weak} shows that the adjoint state $\lopt$ has the weak derivative
	\begin{equation*}
		\dot{\lopt}=-A^\top\lopt+\xopt-x_d\in L^2(0,T;\R^n), \qquad \text{i.e.,}  \qquad   \lopt\in\calX.
	\end{equation*}
	Integrating by parts in \eqref{eq:ODE_adjoint_weak} for general $h\in\calX$ then yields $\lopt(T)=-\nabla\varphi(\xopt(T))$ and $\lopt_0=\lopt(0)$. Consequently, $\popt=(\lopt,\lopt(0))$, and the optimality system becomes the forward-backward system
	\begin{equation}\label{eq:ODE_fb}
		\left\{
		\begin{aligned}
			\dot{\xopt}&=A\xopt+B\uopt, && \xopt(0)=x_0,\\
			-\dot{\lopt}-A^\top\lopt&=-(\xopt-x_d), && \lopt(T)=-\nabla\varphi(\xopt(T)),\\
			\uopt&=P_{U_\mathrm{ad}}\big(\tfrac1\alpha B^\top\lopt\big) && \text{a.e.\ in }(0,T),
		\end{aligned}\right.
	\end{equation}
	in which state and adjoint state enjoy the same regularity, i.e., $\xopt,\lopt\in\calX$. In operator form, the adjoint equation reads $\Ad\lopt=(-(\xopt-x_d),-\nabla\varphi(\xopt(T)))$ with the backward operator
	\begin{equation*}
		\Ad:\calX\to L^2(0,T;\R^n)\times\R^n,\qquad \Ad h=\stack{-\dot h-A^\top h}{h(T)},
	\end{equation*}
	which, like $\calA$, is an isomorphism.

	The additional regularity of the multiplier does not stem from the particular structure of the ODE, but from the regularity of the state cost. Indeed, $J=f\circ E$ with the dense embedding
	\begin{equation*}
		E:\calX\to\calH\coloneqq L^2(0,T;\R^n)\times\R^n,\qquad Ex=\stack{x}{x(T)},
	\end{equation*}
	and $f(h,h_T)\coloneqq\frac12\|h-x_d\|^2_{L^2(0,T;\R^n)}+\varphi(h_T)$. By the chain rule, $DJ(\xopt)=E'Df(E\xopt)$ with $Df(E\xopt)=(\xopt-x_d,\nabla\varphi(\xopt(T)))\in\calH$, where we identify $\calH$ with its dual space. Hence, the right-hand side of the adjoint equation is more regular than a general element of $\calX'$, and so is its solution. In Section~\ref{sec:balancing}, we show that this mechanism carries over to optimal control problems governed by general operator equations whose state cost is of the form $J=f\circ E$ with a dense embedding $E$ of the state space.
\section{Problem setting and optimality conditions}\label{sec:abstract_ocp}
	In this section, we introduce the abstract setting of the optimal control problem \eqref{eq:OCP_intro} in Subsection~\ref{subsec:setting}, and we show existence of a unique optimal solution together with first-order optimality conditions in Subsection~\ref{subsec:optcond}. The resulting optimality system is the starting point for the balancing of regularity in Section~\ref{sec:balancing}.
\subsection{Abstract setting}\label{subsec:setting}
	We now consider an abstract setting which covers the example of Section~\ref{sec:ode_example} as well as evolution equations governed by $C_0$-semigroups, but does not rely on their particular structure. For Banach spaces $Y$ and $Z$, we denote the space of bounded linear operators from $Y$ to $Z$ by $L(Y,Z)$ and the \emph{Banach space adjoint} of $B\in L(Y,Z)$ by $B'\in L(Z',Y')$, defined by
	\begin{equation*}
		\langle B'z',y\rangle_{Y',Y} = \langle z',By\rangle_{Z',Z}, \qquad y\in Y,\quad z'\in Z'.
	\end{equation*}
	For a reflexive Banach space $Z$, we identify the bidual space $Z''$ with $Z$ via the canonical isometric isomorphism, such that $B'\in L(Z,Y')$ for $B\in L(Y,Z')$. To this end, let $\calX$ be a Banach space, let $\calP$ be a reflexive Banach space, let $\calU$ be a Hilbert space, which we identify with its dual via the Riesz isomorphism, let $\calA\in L(\calX,\calP')$ be an isomorphism, $\calB\in L(\calU,\calP')$ and $b\in\calP'$ an abstract datum. In particular, $\calA'\in L(\calP,\calX')$ and $\calB'\in L(\calP,\calU)$. Even if $\calP$ is a Hilbert space, we do not identify $\calP$ with its dual space, as the natural choice of $\calP$ depends on the application, e.g., $\calP=\calX$ for weak formulations of elliptic problems. We consider the optimal control problem
	\begin{equation}\label{eq:OCP}
		\min_{(x,u)\in \calX\times\calU} J(x) + \tfrac\alpha 2 \|u\|_{\calU}^2 +\theta(u) \quad \mathrm{s.t.} \quad \calA x- \calB u = b
	\end{equation}
	where
	\begin{enumerate}[label=(\roman*)]
		\item $J: \calX \to \R$ is a convex and Fr\'echet differentiable cost function,
		\item $\alpha > 0$ is a regularization parameter,
		\item $\theta:\calU\to\R\cup\{+\infty\}$ is a proper\footnote{The \textit{effective domain} of a function $\phi: \calU \to \R \cup \{+\infty\}$ is
		$\domain{\phi}\coloneqq \left\{u \in \calU \, \middle| \, \phi(u) < + \infty \right\}$ and $\phi$ is \textit{proper}, if $\domain{\phi}\neq \emptyset$.}, convex and lower semicontinuous\footnote{
		A function \(g:\calU\to\R\cup\{+\infty\}\) is called \emph{(weakly) lower semicontinuous (lsc)} if
		\begin{equation*}
			g(u)\leq \liminf_{n\to\infty} g(u_n) \, \qquad \forall (u_n)_{n \in \N} \subset \calU: u_n\to u (u_n \rightharpoonup u).
		\end{equation*}
	}, possibly nonsmooth control cost.
	\end{enumerate}
	Here, Fr\'echet differentiability means that, for every $x\in\calX$, there exists $DJ(x)\in\calX'$ such that
	\begin{equation*}
		\lim_{\|h\|_\calX\to0}\frac{| J(x+h)-J(x)-DJ(x)h|}{\|h\|_\calX}=0.
	\end{equation*}
	For instance, control constraints $u\in\Uad$ for a nonempty, closed and convex set $\Uad\subset\calU$ are covered by $\theta=\iota_{\Uad}$. For $\calU=L^2(0,T;\R^m)$, sparsity-promoting costs are covered by $\theta=\beta\|\cdot\|_{L^1(0,T;\R^m)}$ with $\beta>0$. For a proper, convex and lower semicontinuous function $\theta:\calU\to\R\cup\{+\infty\}$ and $\gamma>0$, the \emph{proximal operator} is given by
	\begin{equation*}
		\prox_{\gamma\theta}(z)\coloneqq\operatorname*{arg\,min}_{u\in\calU}\Big(\theta(u)+\tfrac{1}{2\gamma}\|u-z\|_\calU^2\Big),\qquad z\in\calU.
	\end{equation*}
	For the indicator function $\iota_C$ of a nonempty, closed and convex set $C\subset\calU$, we have $\prox_{\gamma\iota_C}=P_C$ with the \emph{orthogonal projection} $P_C:\calU\to C$ onto $C$, cf.\ \cite[Ex.~12.25]{BausComb2011}.
\subsection{Existence of optimal solutions and optimality conditions}\label{subsec:optcond}
To keep the article self-contained, we are now proving existence of optimal controls and provide first-order optimality conditions.
\begin{thrm}\label{thm:existence_optimal_sol}
	The optimal control problem \eqref{eq:OCP} has a unique solution $(\xopt, \uopt) \in \calX \times \calU$. Moreover, there exists a unique Lagrange multiplier $\popt \in \calP$ such that
	\begin{equation}\label{eq:optsys}
		\left\{
		\begin{aligned}
			\calA\xopt-\calB\uopt&=b &&\textup{in }\calP',\\
			\calA'\popt&=-DJ(\xopt) &&\textup{in }\calX',\\
			\uopt&=\prox_{\theta/\alpha}\big(\tfrac{1}{\alpha}\calB'\popt\big) &&\textup{in }\calU.
		\end{aligned}
		\right.
	\end{equation}
\end{thrm}
\begin{proof}
	Using the bijectivity of $\calA$, we equivalently rewrite the optimal control problem \eqref{eq:OCP} by the (unconstrained) optimization problem
	\begin{equation}\label{eq:composite}
		\min_{u \in \calU} g(u) \coloneqq J(\calA^{-1}(\calB u + b)) + \tfrac{\alpha}{2} \|u \|_\calU^2 + \theta(u)
	\end{equation}
	where $\theta$ denotes the nonsmooth part of the objective functional $g$. Recall that $J$ and $\theta$ are convex functions. Indeed, by $\alpha$-strong convexity of the control cost and as $g$ is proper and lower semicontinuous, \eqref{eq:composite} defines an $\alpha$-strongly convex optimization problem, which has a unique solution $\uopt \in \calU$, \cite[Chap.~II, Prop.~1.2]{EkelTema99}. Thus, $(\xopt, \uopt) \in \calX \times \calU$ where $\xopt \coloneqq \calA^{-1}(\calB \uopt + b)$ defines a unique solution of \eqref{eq:OCP}. We now deduce first-order optimality conditions associated with \eqref{eq:OCP}. For $\uopt \in \calU$ to be a global minimizer of \eqref{eq:composite}, it is necessary and sufficient that $0 \in \partial g(\uopt)$, \cite[Thm. 16.3]{BausComb2011}, where $\partial g(\uopt)$ denotes the subgradient\footnote{Let $g:\calU \to \R \cup \{+\infty\}$ be a convex lsc function. The \textit{subgradient} of $g$ at $u_0 \in \calU$ is the set-valued operator $\partial g : \calU \rightrightarrows \calU'$ with
	\begin{align*}
		\partial g(u_0) \coloneqq \left\{ v \in \calU' \,\middle|\, g(u)\geq g(u_0) +\langle v,u-u_0\rangle_{\calU',\calU} \ \text{ for all}\,u\in\calU \right\}.
	\end{align*}} of $g$ at $\uopt$. Precisely,
	\begin{equation}\label{eq:adjoint_eq}
		\uopt = \underset{u \in \calU}{\operatorname{arg \, min}}\  g \quad \text{if and only if} \quad 0 \in \partial g(\uopt).
	\end{equation}
	Because the effective domain of the mapping $u \mapsto J(\calA^{-1}(\calB u +b)) +  \tfrac{\alpha}{2} \|u \|_\calU^2 \eqqcolon K(u)$ satisfies $\domain{K}=\calU$, the sum rule for subgradients \cite[Cor.~16.48]{BausComb2011} yields
	\begin{equation}\label{eq:sum_of_subgr}
		\partial g(u)  = \partial  K(u) + \partial \theta(u).
	\end{equation}
	Note that the Fr\'echet derivative $\nabla K(u) \in L(\calU,\R)$ of $K$ exists at each $u \in \calU$. Hence, $\partial K(u)=\{\nabla K(u)\}$ by \cite[Prop.~17.31]{BausComb2011}. Define the \textit{control-to-state} map $S : \calU \to \calX$ by $S(u)\coloneqq \calA^{-1}(\calB u + b)$.
	Then, by the chain rule for Fr\'echet derivatives, for any direction $h\in \calU$,
	\begin{equation*}
		DK(u)[h] = DJ(S(u)) \bigl[\calA^{-1}\calB h\bigr] + \alpha \langle u,h\rangle_{\calU}.
	\end{equation*}
	Hence
	\begin{align*}
		DK(u) =(\calA^{-1}\calB)' DJ(S(u))+\alpha u =\calB'(\calA^{-1})'DJ(S(u))+ \alpha u.
	\end{align*}
	Thus, by a combination of \eqref{eq:adjoint_eq} and \eqref{eq:sum_of_subgr}, the optimal control $\uopt \in \calU$ satisfies
	\begin{equation*}
		-\tfrac{1}{\alpha} \calB'(\calA^{-1})'
		DJ(S(\uopt)) - \uopt \in \tfrac{1}{\alpha}\partial \theta(\uopt)=\partial \tfrac{1}{\alpha}\theta(\uopt).
	\end{equation*}
	As a consequence, $\uopt$ fulfills the fixed-point equation
	\begin{equation*}
		\uopt = (I+\partial  \tfrac{1}{\alpha}\theta)^{-1} \left(-\tfrac{1}{\alpha} \calB'(\calA^{-1})' DJ(S(\uopt))\right) = \prox_{\theta/\alpha}\left(-\tfrac{1}{\alpha} \calB'(\calA^{-1})' DJ(S(\uopt))\right) ,
	\end{equation*}
	see \cite[Prop. 16.44]{BausComb2011}. Define $\popt \in \calP$ by
	\begin{equation*}
		\popt \coloneqq - (\calA^{-1})' DJ(S(\uopt)) = -(\calA^{-1})' DJ(\xopt),
	\end{equation*}
	where $(\calA^{-1})'\in L(\calX',\calP)$ by the reflexivity of $\calP$, and note that this is equivalent to the second line of \eqref{eq:optsys}. Consequently,
	\begin{equation}\label{eq:optimal_u}
		\uopt=\prox_{\theta/\alpha}\left(\tfrac{1}{\alpha}\calB'  \popt\right),
	\end{equation}
	which is the third line of \eqref{eq:optsys}.
    Finally, the first line of \eqref{eq:optsys} is the state equation in \eqref{eq:OCP}.
\end{proof}
\section{Balancing regularity via dual companions}\label{sec:balancing}
	The optimality system \eqref{eq:optsys} involves the Banach space adjoint $\calA':\calP\to\calX'$. Hence, the adjoint equation is posed in the dual space $\calX'$, whereas the state equation is posed in the data space $\calP'$. In this section, we show that this imbalance can be removed if the state cost enjoys additional regularity, i.e., if $J=f\circ E$ for a dense embedding $E$ of $\calX$ into a reflexive Banach space $\calH$ and a Fr\'echet differentiable function $f$ on $\calH$. In this case, the adjoint equation can be expressed by the Banach space adjoint of an unbounded realization of $\calA$ in $\calH$. For evolution equations, this yields the well-known interpretation of the adjoint equation as an evolution equation posed backward in time, see, e.g., \cite{troltzsch2010optimal} and \cite[Sec.~3.2]{schiela2013concise}.

	For the reader's convenience, we briefly recall some notions concerning unbounded operators and embeddings. Let $Y$ and $Z$ be Banach spaces. A linear, possibly unbounded operator $A:Y\supset\dom A\to Z$ is {\em closed} if its {\em domain} $\dom A$ is complete with respect to the {\em graph norm}
	\[\|y\|_{\dom A}\coloneqq\big(\|y\|_Y^2+\|Ay\|_Z^2\big)^{1/2}.\]
	If $A$ is densely defined, its \textit{Banach space adjoint} $A':Z'\supset\dom{A'}\to Y'$ is defined on
	\[
	\dom{A'}=\left\lbrace z'\in Z' : \exists\, y'\in Y' \text{ s.t.\ }\forall\,y\in\dom A:\;\langle z',Ay\rangle_{Z',Z}=\langle y',y\rangle_{Y',Y} \right\rbrace
	\]
	by $A'z'=y'$, where the density of $\dom A$ in $Y$ implies that $y'$ is uniquely determined by $z'$. The operator $A'$ is closed, and if $Y$ and $Z$ are reflexive and $A$ is closed, then $\dom{A'}$ is dense in $Z'$, see \cite[Thm.~III.5.29]{kato2013perturbation}.

	If $E\in L(Y,Z)$ is injective with dense range, then $E'\in L(Z',Y')$ is injective as well, since $E'z'=0$ implies $\langle z',Ey\rangle_{Z',Z}=0$ for all $y\in Y$ and hence $z'=0$ by the density of $E(Y)$ in $Z$. We then write $Y\hookrightarrow^{E}Z$ and $Z'\hookrightarrow^{E'}Y'$. If $Z$ is a Hilbert space, identified with its dual space via the Riesz isomorphism, this yields the \emph{Gelfand triplet} $Y\hookrightarrow^{E}Z\hookrightarrow^{E'}Y'$ with pivot space $Z$.

    The following result is the key step towards balancing the optimality system. It shows that an adjoint equation $\calA'\popt=-DJ(\xopt)$, which is a priori posed in $\calX'$, can be rewritten as an equation in $\calH'$ whenever its right-hand side lies in $E'(\calH')$, as is the case for $J=f\circ E$. This equation is governed by the Banach space adjoint of the realization $\calA_E$ of $\calA$ in $\calH$, and the multiplier belongs to the domain of $\calA_E'$. Hence, the regularity of the multiplier is determined by $\dom{\calA_E'}$, which can be identified in applications by integration by parts, cf.\ Proposition~\ref{prop:companion_unique}.
\begin{lmm}\label{lem:schiela_38_equiv}
	Let $\calX$ be a Banach space, let $\calP$ and $\calH$ be reflexive Banach spaces, let $E\in L(\calX,\calH)$ be injective with dense range, i.e., $\calX\hookrightarrow^{E}\calH$ and $\calH'\hookrightarrow^{E'}\calX'$, and let $\calA\in L(\calX,\calP')$ be an isomorphism. Define the possibly unbounded operator $\calA_E:\calH\supset\dom{\calA_E}\to\calP'$ by
	\begin{equation*}
		\dom{\calA_E} = E(\calX), \qquad \calA_E (Ex) = \calA x.
	\end{equation*}
	Then $\calA_E$ is well-defined, densely defined and closed, and its Banach space adjoint $\calA_E':\calP\supset\dom{\calA_E'}\to\calH'$ is well-defined. Moreover, for $p\in\calP$, the following statements are equivalent:
	\begin{enumerate}[label=(\roman*)]
		\item $\calA'p\in E'(\calH')$, i.e., there exists $\eta\in\calH'$ such that
		\begin{equation}\label{eq:range_restr}
			\calA' p = E' \eta,
		\end{equation}
		\item $p \in \dom{\calA_E'}$.
	\end{enumerate}
	In this case,
	\begin{equation}\label{eq:calA_adjoint}
		\calA_E' p = (E')^{-1}\calA'p,
	\end{equation}
	where $(E')^{-1}: E'(\calH') \to \calH'$ denotes the inverse of $E':\calH' \to E'(\calH')$.
\end{lmm}
\begin{proof}
	Let $x_1,x_2\in\calX$ with $Ex_1=Ex_2$. As $E$ is injective, this implies $x_1=x_2$ and hence $\calA x_1=\calA x_2$, i.e., $\calA_E$ is well-defined. Moreover, $\calA_E$ is densely defined, as $E(\calX)$ is dense in $\calH$. To show that $\calA_E$ is closed, let $(x_n)_{n\in\N}\subset\calX$ with
	\begin{equation*}
		Ex_n\to h \quad\text{in }\calH, \qquad \calA_E(Ex_n)=\calA x_n\to z \quad\text{in }\calP'.
	\end{equation*}
	As $\calA$ is an isomorphism, $x_n=\calA^{-1}\calA x_n\to\calA^{-1}z$ in $\calX$, and the continuity of $E$ yields $Ex_n\to E\calA^{-1}z$ in $\calH$. Hence, $h=E\calA^{-1}z\in\dom{\calA_E}$ and $\calA_E h=\calA\calA^{-1}z=z$, i.e., $\calA_E$ is closed. As $\calA_E$ is densely defined and $\calP''=\calP$ by reflexivity, its Banach space adjoint $\calA_E':\calP\supset\dom{\calA_E'}\to\calH'$ is well-defined.

	We now show that (ii) implies (i). Let $p\in\dom{\calA_E'}$ and $\eta\coloneqq\calA_E'p\in\calH'$. By definition of the adjoint, $\langle\calA_E g,p\rangle_{\calP',\calP}=\langle\eta,g\rangle_{\calH',\calH}$ for all $g\in E(\calX)$. Choosing $g=Ex$ yields
	\begin{equation*}
		\langle E'\eta, x \rangle_{\calX', \calX}=\langle \eta, Ex \rangle_{\calH',\calH}=  \langle \calA_E Ex, p  \rangle_{\calP',\calP} = \langle \calA  x, p \rangle_{\calP',\calP} = \langle \calA'p, x \rangle_{\calX',\calX}
	\end{equation*}
	for all $x \in \calX$, i.e., \eqref{eq:range_restr} holds.

	Conversely, let $\calA'p = E'\eta$ for some $\eta \in \calH'$. Then
	\begin{equation*}
		\langle \calA_E Ex, p  \rangle_{\calP',\calP} = \langle \calA x, p \rangle_{\calP',\calP} = \langle \calA'p, x \rangle_{\calX',\calX} =   \langle E'\eta, x \rangle_{\calX', \calX}=\langle \eta, Ex \rangle_{\calH',\calH}
	\end{equation*}
	for all $x \in \calX$, i.e., $p \in \dom{\calA_E'}$ with $\calA_E'p=\eta$.

	In both cases, we have $\calA_E'p=\eta$ and $\calA'p=E'\eta$, such that \eqref{eq:calA_adjoint} follows from the injectivity of $E'$. This yields the assertion.
\end{proof}
	Lemma~\ref{lem:schiela_38_equiv} shows that an adjoint equation $\calA'p=E'\eta$ with right-hand side in $E'(\calH')$ can equivalently be written as $\calA_E'p=\eta$ with the unbounded operator $\calA_E'$. Thus, the domain of $\calA_E'$ plays the role of a solution space for adjoint equations with such regular right-hand sides. This motivates the following construction of the dual companion associated with a Banach space $\calX$, an embedding $E$ and an isomorphism $\calA$.
\begin{prpstn}[Dual companions]\label{prop:dual_companion}
	Let the assumptions of Lemma~\Rref{lem:schiela_38_equiv} hold, and let $\calA_E: \calH \supset \dom{\calA_E} \to \calP'$ be the operator defined therein. Define the space
	\begin{equation*}
		\Xd \coloneqq \dom{\calA_E'}= \left\{p \in \calP : \calA' p \in \ran{E'} \right\} , \qquad \| \lambda \|_{\Xd}^2 \coloneqq \| \lambda \|_\calP^2 + \| \calA_E'\lambda \|_{\calH'}^2
	\end{equation*}
	together with the embedding
	\begin{equation*}
		\Ed : \Xd \hookrightarrow \calP , \qquad \Ed \lambda = \lambda
	\end{equation*}
	and the operator
	\begin{equation*}
		\Ad : \Xd \to \calH', \qquad \Ad \lambda = (E')^{-1} \calA' \lambda.
	\end{equation*}
	Then the following statements hold:
	\begin{enumerate}[label=(\roman*)]
		\item $\Xd$ is a Banach space. If, in addition, $\calH$ and $\calP$ are Hilbert spaces, then so is $\Xd$.
		\item $\Ed$ is continuous and injective with dense range.
		\item $\Ad$ is an isomorphism.
		\item The identity
		\begin{equation}\label{eq:dual_companion_identity}
			\calA' \Ed = E'\Ad : \Xd \to \calX'
		\end{equation}
		holds.
	\end{enumerate}
	We call the triple $(\Xd, \Ed,\Ad)$ the \textit{dual companion of $(\calX, E,\calA)$}.
\end{prpstn}
\begin{proof}
	As the Banach space adjoint of a densely defined operator, $\calA_E'$ is closed. Hence, its domain $\Xd$ endowed with the graph norm is a Banach space, and the graph norm is induced by an inner product if $\calH$ and $\calP$ are Hilbert spaces. This shows (i).

	As $\|\lambda\|_\calP\leq\|\lambda\|_{\Xd}$ for all $\lambda\in\Xd$, the inclusion $\Ed$ is continuous and injective. Moreover, $\calA_E$ is densely defined and closed by Lemma~\ref{lem:schiela_38_equiv}, and $\calH$ and $\calP'$ are reflexive. Hence, $\Xd=\dom{\calA_E'}$ is dense in $\calP$, see \cite[Thm.~III.5.29]{kato2013perturbation}, which shows (ii).

	For every $\lambda\in\Xd$, the definition of $\Ad$ yields
	\begin{equation*}
		E'\Ad \lambda = E'(E')^{-1}\calA'\lambda = \calA'\lambda=\calA'\Ed \lambda,
	\end{equation*}
	which proves (iv).

	It remains to show (iii). By \eqref{eq:calA_adjoint}, we have $\Ad \lambda=\calA_E'\lambda$ and hence $\|\Ad \lambda\|_{\calH'}\leq\|\lambda\|_{\Xd}$ for all $\lambda\in\Xd$, i.e., $\Ad$ is bounded. As $\calA$ is an isomorphism, so is $\calA'\in L(\calP,\calX')$. If $\Ad \lambda=0$, then $\calA'\lambda=E'\Ad \lambda=0$ by (iv), and hence $\lambda=0$, i.e., $\Ad$ is injective. For $\eta\in\calH'$, the element $\lambda\coloneqq(\calA')^{-1}E'\eta\in\calP$ satisfies $\calA'\lambda=E'\eta\in\ran{E'}$, i.e., $\lambda\in\Xd$, and $\Ad \lambda=\eta$. Hence, $\Ad$ is bijective with $(\Ad)^{-1}=(\calA')^{-1}E'$, and the bounded inverse theorem yields that $\Ad$ is an isomorphism. This yields the assertion.
\end{proof}
We now show that the construction of dual companions applied to the optimality conditions \eqref{eq:optsys} yields an optimality system with balanced regularity, that is, the state and the adjoint state $\lopt$, which represents the multiplier via $\popt=\Ed\lopt$, enjoy the same kind of regularity.
\begin{thrm}\label{thm:adjoint_regularity_general}
	Consider the optimal control problem \eqref{eq:OCP} in the setting of Section~\Rref{sec:abstract_ocp} and let $E:\calX\hookrightarrow \calH$ be a continuous and injective embedding with dense range into a reflexive Banach space $\calH$. Suppose that the state cost satisfies
	\begin{equation}\label{eq:general_cost_embedding}
		J(x)=f(Ex),
	\end{equation}
	for some Fr\'echet differentiable function $f:\calH\to\R$ with derivative $Df(h)\in\calH'$ at $h\in\calH$. Let $(\Xd, \Ed,\Ad)$ be the dual companion associated with $(\calX, E,\calA)$. Then, $(\xopt,\uopt,\popt)\in\calX\times\calU\times\calP$
	satisfies the optimality system \eqref{eq:optsys} if and only if
	\begin{equation*}
		\popt\in \Ed(\Xd), \qquad \text{i.e.,}  \qquad \popt=\Ed\lopt
		\qquad
		\text{for a unique }\lopt\in\Xd,
	\end{equation*}
	and $(\xopt,\uopt,\lopt)\in\calX\times\calU\times\Xd$ satisfies
	\begin{equation}\label{eq:general_balanced_optsys_direct}
		\left\{
		\begin{aligned}
			\calA\xopt-\calB\uopt &=b &&\textup{in }\calP',\\
			\Ad\lopt &=- Df(E\xopt) &&\textup{in }\calH',\\
			\uopt&=\prox_{\theta/\alpha}\big(\tfrac{1}{\alpha}\calB'\Ed\lopt\big) &&\textup{in }\calU.
		\end{aligned}
		\right.
	\end{equation}
\end{thrm}
\begin{proof}
	Let $(\xopt,\uopt,\popt)$ satisfy \eqref{eq:optsys}. By \eqref{eq:general_cost_embedding} and the chain rule for Fr\'echet derivatives,
	\begin{equation*}
		DJ(x)=E'Df(Ex), \qquad x\in\calX.
	\end{equation*}
	Hence, the second line of \eqref{eq:optsys} yields
	\begin{equation*}
		\calA'\popt =
		-E'Df(E\xopt).
	\end{equation*}
	In particular, $\calA'\popt\in\ran{E'}$. By Lemma~\ref{lem:schiela_38_equiv}, this is equivalent to $ \popt\in\dom{\calA_E'} = \Ed(\Xd)$, and
	\begin{equation*}
		\calA_E'\popt= -Df(E\xopt).
	\end{equation*}
	Since $\Ed: \Xd \to \dom{\calA_E'}$ is an isomorphism, there exists a unique $\lopt\in\Xd$ such that $ \popt=\Ed\lopt$.
	Consequently,
	\begin{equation*}
		\Ad\lopt = \calA_E'\Ed\lopt = -Df(E\xopt),
	\end{equation*}
	which is the second line of \eqref{eq:general_balanced_optsys_direct}. The first and third lines of \eqref{eq:general_balanced_optsys_direct} follow from those of \eqref{eq:optsys} and $\popt=\Ed\lopt$.
	Conversely, let $\lopt\in\Xd$ be such that $(\xopt,\uopt,\lopt)$ satisfies \eqref{eq:general_balanced_optsys_direct} and set $\popt\coloneqq \Ed\lopt$. Then \eqref{eq:dual_companion_identity} yields
	\begin{equation*}
		\calA'\popt=E'\Ad\lopt=-E'Df(E\xopt)=-DJ(\xopt),
	\end{equation*}
	which is the second line of \eqref{eq:optsys}. The first and third lines of \eqref{eq:optsys} follow from those of \eqref{eq:general_balanced_optsys_direct} and $\popt=\Ed\lopt$. This yields the assertion.
\end{proof}
\begin{figure}[t]
\centering
\begin{minipage}[b]{0.41\textwidth}\centering
\begin{tikzpicture}[
    >={Stealth[length=2mm]},
    every node/.style={font=\small},
    sp/.style={draw, rounded corners=3pt, minimum width=1.05cm, minimum height=0.75cm, align=center, thick, fill=gray!7},
    st/.style={blue!70!black, very thick},
    ad/.style={red!70!black, very thick},
    lab/.style={font=\scriptsize, align=center},
    dual/.style={densely dotted, gray!60!black, thick},
    scale=0.85, transform shape]
\useasboundingbox (-2.6,-1.6) rectangle (5.0,4.4);
\coordinate (X) at (0,0);
\draw[dashed, blue!70!black, thick, fill=blue!10] (X) ellipse (1.0 and 0.5);
\node at ($(X)+(-0.5,0.1)$) {$\calX$};
\fill[blue!70!black] ($(X)+(0.2,-0.15)$) circle (1.4pt) node[right, lab] {$\xopt$};
\node[sp] (Pd) at (0,3.2) {$\calP'$};
\draw[->, st] ($(X)+(0,0.5)$) -- node[left, lab, blue!70!black] {$\calA$\\ state\\ equation} (Pd.south);
\coordinate (P) at (3.4,3.2);
\draw[thick, fill=red!3] (P) ellipse (1.4 and 0.8);
\node at ($(P)+(1.1,0.8)$) {$\calP$};
\fill[red!70!black] ($(P)+(0,-0.1)$) circle (1.4pt) node[right, lab] {$\popt$};
\node[sp] (Xp) at (3.4,0) {$\calX'$};
\draw[->, ad] ($(P)+(0,-0.8)$) -- node[right, lab, red!70!black] {$\calA'$\\ adjoint\\ equation} (Xp.north);
\draw[dual] ($(X)+(1.0,0)$) -- node[below, lab, gray!60!black] {dual} (Xp.west);
\draw[dual] (Pd.east) -- node[above, lab, gray!60!black] {dual} ($(P)+(-1.4,0)$);
\end{tikzpicture}\\[1mm]
{\small (a) Unbalanced regularity}
\end{minipage}\hfill
\begin{minipage}[b]{0.57\textwidth}\centering
\begin{tikzpicture}[
    >={Stealth[length=2mm]},
    every node/.style={font=\small},
    sp/.style={draw, rounded corners=3pt, minimum width=1.05cm, minimum height=0.75cm, align=center, thick, fill=gray!7},
    st/.style={blue!70!black, very thick},
    ad/.style={red!70!black, very thick},
    lab/.style={font=\scriptsize, align=center},
    dual/.style={densely dotted, gray!60!black, thick},
    scale=0.85, transform shape]
\useasboundingbox (-3.4,-1.6) rectangle (7.0,4.4);
\coordinate (H) at (0,0);
\draw[thick, fill=blue!4] (H) ellipse (1.5 and 0.85);
\node at ($(H)+(-1.15,0.8)$) {$\calH$};
\draw[dashed, blue!70!black, thick, fill=blue!10] (H) ellipse (1.05 and 0.52);
\node[lab] at ($(H)+(0,0.18)$) {$E(\calX)$};
\fill[blue!70!black] ($(H)+(-0.25,-0.18)$) circle (1.4pt) node[right, lab] {$E\xopt$};
\node[sp, fill=blue!7] (X) at (-2.6,0) {$\calX$};
\draw[->, thick] (X.east) -- node[below, lab] {$E$ \ \ \ } ($(H)+(-1.05,0)$);
\node[sp] (Pd) at (0,3.2) {$\calP'$};
\draw[->, st] ($(H)+(0,0.52)$) -- node[right, lab, blue!70!black, pos=0.5] {$\calA_E$\\ state\\ equation} (Pd.south);
\draw[->, blue!70!black, thick] (X.north) to[out=90, in=180] node[right, lab, blue!70!black, pos=0.55] {$\calA=\calA_E E$} (Pd.west);
\coordinate (P) at (3.4,3.2);
\draw[thick, fill=red!3] (P) ellipse (1.5 and 0.85);
\node at ($(P)+(1.2,0.8)$) {$\calP$};
\draw[dashed, red!70!black, thick, fill=red!8] (P) ellipse (1.05 and 0.52);
\node[lab] at ($(P)+(0,0.18)$) {$\Ed(\Xd)$};
\fill[red!70!black] ($(P)+(-0.6,-0.2)$) circle (1.4pt) node[right, lab] {$\popt=\Ed\lopt$};
\node[sp, fill=red!6] (Xd) at (6.2,3.2) {$\Xd$};
\draw[->, thick] (Xd.west) -- node[below, lab] {\ \ \ $\Ed$} ($(P)+(1.05,0)$);
\node[sp] (Hd) at (3.4,0) {$\calH'$};
\draw[->, ad] ($(P)+(0,-0.52)$) -- node[left, lab, red!70!black, pos=0.5] {$\calA_E'$\\ adjoint\\ equation} (Hd.north);
\draw[->, red!70!black, thick] (Xd.south) to[out=-90, in=0] node[below right, lab, red!70!black, pos=0.7] {$\Ad=\calA_E'\Ed$} (Hd.east);
\draw[dual] ($(H)+(1.5,0)$) -- node[below, lab, gray!60!black] {dual} (Hd.west);
\draw[dual] (Pd.east) -- node[above, lab, gray!60!black] {dual} ($(P)+(-1.5,0)$);
\end{tikzpicture}\\[1mm]
{\small (b) Balanced regularity}
\end{minipage}
\caption{Optimality system with unbalanced (a) and balanced (b) regularity.}
\label{fig:balancing}
\end{figure}
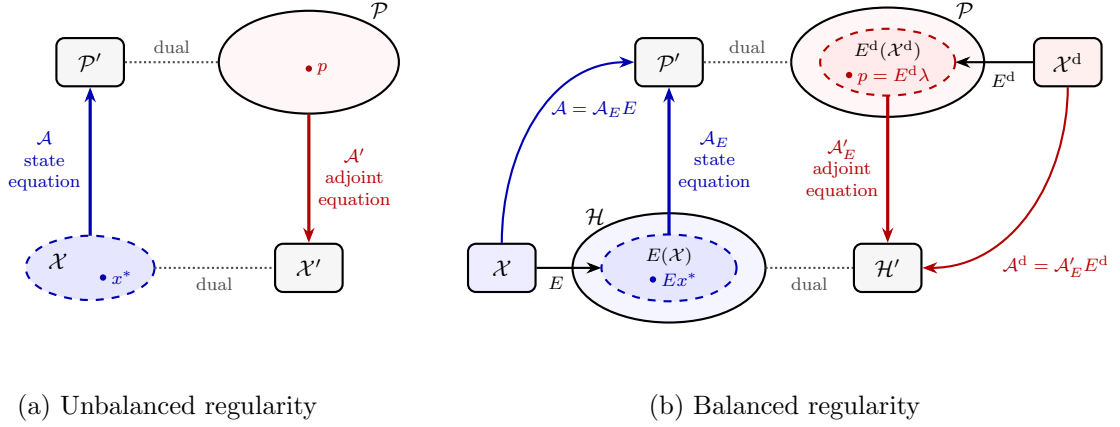

	The space $\Xd$ in Proposition~\ref{prop:dual_companion} is defined abstractly as a subspace of $\calP$. In applications, the main task is to identify $\Xd$ with a concrete solution space, e.g., of a backward evolution equation. The following result shows that this only requires an identity of the form \eqref{eq:dual_companion_identity} for a candidate triple, see \eqref{eq:companion_identity_Y} below, which typically follows from integration by parts, together with the bounded invertibility of the candidate operator.
	\begin{prpstn}[Representations of the dual companion]\label{prop:companion_unique}
		Let the assumptions of Proposition~\Rref{prop:dual_companion} hold and let $(\Xd,\Ed,\Ad)$ be the dual companion of $(\calX,E,\calA)$. Let $\tXd$ be a Banach space, let $\tEd:\tXd\to\calP$ be linear and let $\tAd:\tXd\to\calH'$ be an isomorphism. Then
		\begin{equation}\label{eq:companion_identity_Y}
			\calA'\tEd = E'\tAd
		\end{equation}
		holds if and only if there exists an isomorphism $R:\tXd\to\Xd$ with $\tEd=\Ed R$ and $\tAd=\Ad R$. In this case, $R=(\Ad)^{-1}\tAd$, $\tEd(\tXd)=\Xd=\dom{\calA_E'}$ and $\calA_E'\tEd=\tAd$, and $\tEd$ is continuous and injective. We then call $(\tXd,\tEd,\tAd)$ a \emph{representation} of the dual companion.
	\end{prpstn}
	\begin{proof}
		If $R:\tXd\to\Xd$ is an isomorphism with $\tEd=\Ed R$ and $\tAd=\Ad R$, then \eqref{eq:dual_companion_identity} yields $\calA'\tEd=\calA'\Ed R=E'\Ad R=E'\tAd$. Conversely, let \eqref{eq:companion_identity_Y} hold and let $y\in\tXd$. Then $\calA'\tEd y=E'\tAd y\in\ran{E'}$. Hence, Lemma~\ref{lem:schiela_38_equiv} yields $\tEd y\in\dom{\calA_E'}=\Xd$ and
		\begin{equation*}
			\Ad\tEd y=\calA_E'\tEd y=(E')^{-1}\calA'\tEd y=\tAd y.
		\end{equation*}
		Since $\Ad$ is bijective by Proposition~\ref{prop:dual_companion}, this implies $\tEd y=(\Ad)^{-1}\tAd y=Ry$ with $R\coloneqq(\Ad)^{-1}\tAd$, i.e., $\tEd=\Ed R$ and $\tAd=\Ad R$. As $\tAd$ and $(\Ad)^{-1}$ are isomorphisms, so is $R$. In particular, $\tEd(\tXd)=R(\tXd)=\Xd$, and $\tEd=\Ed R$ is continuous and injective. Finally, any isomorphism $R$ with $\tAd=\Ad R$ is given by $R=(\Ad)^{-1}\tAd$. This yields the assertion.
	\end{proof}
	Thus, the dual companion itself is canonical, whereas its representation in terms of a concrete solution space depends on the identification of $\Xd$.
The construction of the dual companion is symmetric in the sense that the original triple is recovered, up to representation, by forming the dual companion twice.
\begin{crllr}[Duality of dual companions]\label{cor:companion_duality}
	Let the assumptions of Proposition~\Rref{prop:dual_companion} hold and let $(\Xd,\Ed,\Ad)$ be the dual companion of $(\calX,E,\calA)$. Then $(\Xd,\Ed,\Ad)$ satisfies the assumptions of Proposition~\Rref{prop:dual_companion} with the roles of $\calH$ and $\calP$ interchanged, and its dual companion $((\Xd)^{\mathrm d},(\Ed)^{\mathrm d},(\Ad)^{\mathrm d})$ is given by
	\begin{equation*}
		(\Xd)^{\mathrm d}=E(\calX)=\dom{\calA_E},\qquad (\Ed)^{\mathrm d}h=h,\qquad (\Ad)^{\mathrm d}=\calA_E.
	\end{equation*}
	Moreover, $E:\calX\to(\Xd)^{\mathrm d}$ is an isomorphism satisfying
	\begin{equation*}
		E=(\Ed)^{\mathrm d}E,\qquad \calA=(\Ad)^{\mathrm d}E,
	\end{equation*}
	i.e., $(\calX,E,\calA)$ is a representation of the dual companion of $(\Xd,\Ed,\Ad)$ in the sense of Proposition~\Rref{prop:companion_unique}.
\end{crllr}
\begin{proof}
	By Proposition~\ref{prop:dual_companion}, $\Xd$ is a Banach space, $\Ed:\Xd\to\calP$ is continuous and injective with dense range, and $\Ad:\Xd\to\calH'$ is an isomorphism. Hence, $(\Xd,\Ed,\Ad)$ satisfies the assumptions of Proposition~\ref{prop:dual_companion} with $\calH$ and $\calP$ interchanged, where we use that both spaces are reflexive.  We apply Proposition~\ref{prop:companion_unique} with $(\Xd,\Ed,\Ad)$ in place of $(\calX,E,\calA)$, with $\calH$ and $\calP$ interchanged, and with the candidate $(\calX,E,\calA)$ in place of $(\tXd,\tEd,\tAd)$. Here, $E:\calX\to\calH$ is linear and $\calA:\calX\to\calP'$ is an isomorphism, as required. For $x\in\calX$ and $\lambda\in\Xd$, the identity \eqref{eq:dual_companion_identity} yields
    \begin{equation*}
        \begin{array}{rclclcl}
             \langle(\Ad)'Ex,\lambda\rangle_{(\Xd)',\Xd}&=&\langle \Ad \lambda,Ex\rangle_{\calH',\calH}&=&\langle E'\Ad \lambda,x\rangle_{\calX',\calX} &&\\
		&=&\langle\calA'\Ed \lambda,x\rangle_{\calX',\calX}&=&\langle\calA x,\Ed \lambda\rangle_{\calP',\calP}&=&\langle(\Ed)'\calA x,\lambda\rangle_{(\Xd)',\Xd},
        \end{array}
    \end{equation*}
	i.e.,
    \begin{equation*}
        (\Ad)'E=(\Ed)'\calA : \calX \to (\Xd)' ,
    \end{equation*}
    which is the identity \eqref{eq:companion_identity_Y} in this setting. Hence, Proposition~\ref{prop:companion_unique} yields $E(\calX)=(\Xd)^{\mathrm d}$ and that $R\coloneqq((\Ad)^{\mathrm d})^{-1}\calA:\calX\to(\Xd)^{\mathrm d}$ is an isomorphism with $E=(\Ed)^{\mathrm d}R$ and $\calA=(\Ad)^{\mathrm d}R$. Since $(\Ed)^{\mathrm d}$ is the inclusion of $(\Xd)^{\mathrm d}$ into $\calH$, the first identity yields $R=E$, and the second one yields $(\Ad)^{\mathrm d}(Ex)=\calA x=\calA_E(Ex)$ for all $x\in\calX$, i.e., $(\Ad)^{\mathrm d}=\calA_E$.
\end{proof}

\section{Applications}\label{sec:examples_ocp}

	In this section, we illustrate the abstract framework by four classes of examples. In contrast to the motivating example of Section~\ref{sec:ode_example}, where the additional regularity was derived by hand, we now obtain it from Theorem~\ref{thm:adjoint_regularity_general} by identifying the dual companion via Proposition~\ref{prop:companion_unique}. We consider linear elliptic PDEs in weak form in Subsection~\ref{subsec:elliptic}, linear parabolic PDEs in the variational setting of \cite{schiela2013concise} in Subsection~\ref{subsec:parabolic}, abstract Cauchy problems governed by general $C_0$-semigroups, which contain the example of Section~\ref{sec:ode_example}, in Subsection~\ref{subsec:cauchy}, and a wave equation with a cost functional acting on the accumulated state in Subsection~\ref{subsec:wave}.

	\subsection{Linear elliptic PDEs in weak form}\label{subsec:elliptic}
	We start with a stationary problem, for which the additional regularity of the adjoint state and, consequently, of the optimal control is well known, see, e.g., \cite{troltzsch2010optimal,hinze2009optimization}. In the following, we show that this regularity is a direct consequence of our general theory. In the weak formulation considered here, the data space is the dual of the state space, i.e., $\calP=\calX$. Hence, in contrast to evolution equations, no imbalance between state and multiplier occurs, and the dual companion identifies the additional regularity of the adjoint state in space beyond the form domain.

	Let $\Omega\subset\R^d$ be a bounded Lipschitz domain, let $\omega\subset\Omega$ be open, and let $K\in L^\infty(\Omega)^{d\times d}$ be symmetric and uniformly positive definite, $w\in W^{1,\infty}(\Omega)^d$ and $c\in L^\infty(\Omega)$ with $c-\tfrac12\operatorname{div}w\geq0$ a.e.\ in $\Omega$. With the dual space $H^{-1}(\Omega)\coloneqq H^1_0(\Omega)'$, we define $L,L^*:H^1_0(\Omega)\to H^{-1}(\Omega)$ by $\langle Ly,v\rangle=a(y,v)$ and $\langle L^*p,v\rangle=a(v,p)$ with the bilinear form
	\begin{equation*}
		a(y,v)\coloneqq\int_\Omega K\nabla y\cdot\nabla v+(w\cdot\nabla y)\,v+c\,y\,v\,\mathrm{d}x,\qquad y,v\in H^1_0(\Omega),
	\end{equation*}
	i.e., $Ly=-\operatorname{div}(K\nabla y)+w\cdot\nabla y+cy$ and $L^*p=-\operatorname{div}(K\nabla p)-w\cdot\nabla p+(c-\operatorname{div}w)p$ in the sense of distributions. As $\int_\Omega(w\cdot\nabla y)\,y\,\mathrm{d}x=-\tfrac12\int_\Omega(\operatorname{div}w)\,y^2\,\mathrm{d}x$ for $y\in H^1_0(\Omega)$, the form $a$ is coercive by the Poincar\'e inequality, and the Lax--Milgram lemma shows that $L$ and $L^*$ are isomorphisms. Moreover, $L^*$ restricts to an isomorphism from the maximal domain
	\begin{equation*}
		\dom{L^*}\coloneqq\left\{p\in H^1_0(\Omega):L^*p\in L^2(\Omega)\right\},\qquad \|\lambda\|_{\dom{L^*}}\coloneqq\|\lambda\|_{H^1(\Omega)}+\|L^*\lambda\|_{L^2(\Omega)},
	\end{equation*}
	onto $L^2(\Omega)$. If $\Omega$ is convex and $K\in W^{1,\infty}(\Omega)^{d\times d}$, then $\dom{L^*}=H^2(\Omega)\cap H^1_0(\Omega)$ with equivalent norms by elliptic regularity, see \cite[Thm.~3.2.1.2]{grisvard1985elliptic}.

	\smallskip

	\noindent \textbf{The optimal control problem.} For $y_d\in L^2(\Omega)$, $b\in H^{-1}(\Omega)$, $\alpha>0$ and a proper, convex and lower semicontinuous function $\theta_0:\R\to\R\cup\{+\infty\}$, we consider
	\begin{equation}\label{eq:OCP_elliptic}
		\min_{(y,u)\in H^1_0(\Omega)\times L^2(\omega)}\ \tfrac12\|y-y_d\|^2_{L^2(\Omega)}+\tfrac\alpha2\|u\|^2_{L^2(\omega)}+\int_\omega\theta_0(u(x))\,\mathrm{d}x\quad\mathrm{s.t.}\quad Ly=\calB u+b,
	\end{equation}
	where $\calB:L^2(\omega)\to H^{-1}(\Omega)$ denotes the extension by zero. This is \eqref{eq:OCP} with $\calX=\calP=H^1_0(\Omega)$, $\calP'=H^{-1}(\Omega)$, $\calU=L^2(\omega)$, $\calA=L$ and
	\begin{equation*}
		\theta(u)=\int_\omega\theta_0(u(x))\,\mathrm{d}x,
	\end{equation*}
	which is proper, convex and lower semicontinuous on $L^2(\omega)$, see \cite[Cor.~to Thm.~2]{rockafellar1968integrals}. Here, $\calA'=L^*$ and $\calB'p=p|_\omega$, as
    \begin{equation*}
        \langle\calA'p,h\rangle=\langle Lh,p\rangle=a(h,p)=\langle L^*p,h\rangle
    \end{equation*}
    for $h,p\in H^1_0(\Omega)$. Hence, Theorem~\ref{thm:existence_optimal_sol} yields a unique solution $(\yopt,\uopt)$ and a unique multiplier $\popt\in H^1_0(\Omega)$ such that
	\begin{equation}\label{eq:elliptic_weak_adjoint}
		a(h,\popt)=-\langle\yopt-y_d,h\rangle_{L^2(\Omega)}\qquad\forall\,h\in H^1_0(\Omega),\qquad \uopt=\prox_{\theta/\alpha}\big(\tfrac1\alpha\popt|_\omega\big).
	\end{equation}
	Thus, state and multiplier are weak solutions of Dirichlet problems in the same form domain $H^1_0(\Omega)$, i.e., the optimality system has balanced regularity from the outset.

	\smallskip

	\noindent\textbf{Dual companion and additional regularity.} Although no balancing is required here, the dual companion shows that the regularity of the cost functional lifts the multiplier $\popt$ beyond the form domain, i.e., $\popt=\lopt$ with an adjoint state $\lopt\in\dom{L^*}$.
	 To this end, we choose $E:H^1_0(\Omega)\to\calH\coloneqq L^2(\Omega)$, $Ey\coloneqq y$, which is continuous and injective with dense range, and identify $\calH$ with its dual space via the $L^2$-inner product. Then $J=f\circ E$ with $f(h)=\tfrac12\|h-y_d\|^2_{L^2(\Omega)}$ and $Df(h)=h-y_d\in\calH'$.
	\begin{prpstn}\label{prop:elliptic_companion}
		The dual companion of $(H^1_0(\Omega),E,L)$ is given by $\Xd=\dom{L^*}$, the inclusion $\Ed$ of $\dom{L^*}$ into $H^1_0(\Omega)$ and $\Ad=L^*:\dom{L^*}\to L^2(\Omega)$.
	\end{prpstn}
	\begin{proof}
		For $h\in H^1_0(\Omega)$ and $\lambda\in\dom{L^*}$, we have
		\begin{equation*}
			(\calA'\lambda)(h)=a(h,\lambda)=\langle L^*\lambda,h\rangle_{L^2(\Omega)}=\int_\Omega h\,L^*\lambda\,\mathrm{d}x=(E'L^*\lambda)(h),
		\end{equation*}
		i.e., the identity \eqref{eq:companion_identity_Y} holds with $\tXd=\dom{L^*}$, $\tEd$ the inclusion into $H^1_0(\Omega)$ and $\tAd=L^*$. As $L^*:\dom{L^*}\to L^2(\Omega)$ is an isomorphism, Proposition~\ref{prop:companion_unique} yields the assertion.
	\end{proof}
	\begin{crllr}\label{cor:elliptic_optsys}
		The triple $(\yopt,\uopt,\popt)$ satisfies \eqref{eq:optsys} if and only if $\popt=\lopt$ for some $\lopt\in\dom{L^*}$ with
		\begin{equation}\label{eq:optsys_elliptic}
			\left\{
			\begin{aligned}
				L\yopt&=\calB\uopt+b &&\textup{in }H^{-1}(\Omega),\\
				L^*\lopt&=-(\yopt-y_d) &&\textup{in }L^2(\Omega),\\
				\uopt(x)&=\prox_{\theta_0/\alpha}\big(\tfrac1\alpha\lopt(x)\big) &&\textup{for a.e. }x\in\omega.
			\end{aligned}
			\right.
		\end{equation}
		In particular, \eqref{eq:optsys_elliptic} has a unique solution, and $\lopt\in H^2(\Omega)\cap H^1_0(\Omega)$ if $\Omega$ is convex and $K\in W^{1,\infty}(\Omega)^{d\times d}$. If, in addition, $d\leq3$, then $\uopt$ has a representative in $C(\overline\omega)$.
	\end{crllr}
	\begin{proof}
		By Theorem~\ref{thm:adjoint_regularity_general} and Proposition~\ref{prop:elliptic_companion}, the triple $(\yopt,\uopt,\popt)$ satisfies \eqref{eq:optsys} if and only if $\popt=\lopt\in\dom{L^*}$ and the first two lines of \eqref{eq:optsys_elliptic} hold together with $\uopt=\prox_{\theta/\alpha}(\tfrac1\alpha\calB'\lopt)=\prox_{\theta/\alpha}(\tfrac1\alpha\lopt|_\omega)$. For $v\in L^2(\omega)$, the function $x\mapsto\prox_{\theta_0/\alpha}(v(x))$ belongs to $L^2(\omega)$, as $\prox_{\theta_0/\alpha}$ is nonexpansive, and minimizes the integrand of $u\mapsto\theta(u)+\tfrac\alpha2\|u-v\|^2_{L^2(\omega)}$ pointwise. Hence, $(\prox_{\theta/\alpha}v)(x)=\prox_{\theta_0/\alpha}(v(x))$ for a.e.\ $x\in\omega$, which yields the third line of \eqref{eq:optsys_elliptic}. Moreover, the optimality conditions in the proof of Theorem~\ref{thm:existence_optimal_sol} are necessary and sufficient, such that \eqref{eq:optsys} and hence \eqref{eq:optsys_elliptic} have exactly one solution.
		 Finally, if $\Omega$ is convex, $K\in W^{1,\infty}(\Omega)^{d\times d}$ and $d\leq3$, then $\lopt\in H^2(\Omega)\hookrightarrow C(\overline\Omega)$, and the composition of $\tfrac1\alpha\lopt$ with the nonexpansive map $\prox_{\theta_0/\alpha}$ is continuous on $\overline\omega$.
	\end{proof}
	The gained regularity stems from the regularity of the cost functional, as the right-hand side $-(\yopt-y_d)$ of the adjoint equation belongs to $L^2(\Omega)$. In particular, if $b\in H^{-1}(\Omega)\setminus L^2(\Omega)$, the adjoint state is even more regular than the state, whereas for $b\in L^2(\Omega)$, the same argument applied to $L$ shows that the state belongs to the maximal domain of $L$ and state and adjoint state enjoy the same regularity.

	\subsection{Linear parabolic PDEs}\label{subsec:parabolic}
    In this subsection, we consider linear parabolic PDEs in the variational setting of \cite{schiela2013concise} and show that the a posteriori lifting of the adjoint state in \cite[Prop.~3.8]{schiela2013concise} is a special case of Theorem~\ref{thm:adjoint_regularity_general}. For simplicity, we restrict ourselves to Hilbert spaces.
	To this end, let $V$ and $H$ be separable Hilbert spaces such that $V\hookrightarrow H$ continuously and densely, and consider the Gelfand triplet $V\hookrightarrow H\simeq H'\hookrightarrow V'$. Here, the time derivative of $y\in L^2(0,T;V)$ is understood in the distributional sense. That is, we write $y'=w$ for some $w\in L^2(0,T;V')$ if
	\begin{equation*}
		\int_0^T\varphi'(t)\,y(t)\,\mathrm{d}t=-\int_0^T\varphi(t)\,w(t)\,\mathrm{d}t\qquad\text{in }V'\text{ for all }\varphi\in C_c^\infty(0,T),
	\end{equation*}
	where $y$ is regarded as a $V'$-valued function via the Gelfand triplet. Equivalently, $\tfrac{\mathrm{d}}{\mathrm{d}t}\langle y(\cdot),v\rangle_H=\langle w(\cdot),v\rangle_{V',V}$ holds in the sense of scalar distributions on $(0,T)$ for all $v\in V$, see \cite[Prop.~23.20]{Zeidler1990}. We denote the duality pairing between $L^2(0,T;V')$ and $L^2(0,T;V)$ by $\langle\cdot,\cdot\rangle$. Consider the space
	\begin{equation*}
		W([0,T])\coloneqq\left\{y\in L^2(0,T;V): y'\in L^2(0,T;V')\right\}
	\end{equation*}
	endowed with the norm $\|y\|_{W([0,T])}\coloneqq\|y\|_{L^2(0,T;V)}+\|y'\|_{L^2(0,T;V')}$.
	As is well known, $W([0,T])\hookrightarrow C([0,T];H)$ continuously and the integration by parts formula
	\begin{equation}\label{eq:ibp_W}
		\langle y',\lambda\rangle+\langle \lambda',y\rangle=\langle y(T),\lambda(T)\rangle_H-\langle y(0),\lambda(0)\rangle_H,\qquad y,\lambda\in W([0,T]),
	\end{equation}
	holds, see, e.g., \cite[Prop.~23.23]{Zeidler1990}. Further, let $\Lambda:L^2(0,T;V)\to L^2(0,T;V')$ be linear and continuous such that $\langle\Lambda y,y\rangle\geq\gamma\|y\|^2_{L^2(0,T;V)}$ for some $\gamma>0$ and all $y\in L^2(0,T;V)$, e.g., $(\Lambda y)(t)=a(y(t),\cdot)$ for a bounded and coercive bilinear form $a:V\times V\to\R$. By \cite[Thm.~3.4]{schiela2013concise},
	\begin{equation*}
		\calA:W([0,T])\to\calP'= L^2(0,T;V')\times H,\qquad \calA y\coloneqq\stack{y'+\Lambda y}{y(0)},
	\end{equation*}
	is an isomorphism, where $\calP\coloneqq L^2(0,T;V)\times H$. Hence, a multiplier is an element $p=(\lambda,\lambda_0)\in\calP$, and
	\begin{equation*}
		(\calA'p)(y)=\langle\calA y,p\rangle_{\calP',\calP}=\langle y'+\Lambda y,\lambda\rangle+\langle y(0),\lambda_0\rangle_H,\qquad y\in W([0,T]).
	\end{equation*}
	In other words, $(\lambda,\lambda_0)$ is the adjoint state of \cite[Thm.~1.1]{schiela2013concise}.

	\smallskip

	\noindent\textbf{The optimal control problem.} As in \cite[Sec.~4]{schiela2013concise}, we consider the control of the initial value, i.e., $\calU=H$, $\calB u=(0,u)$ and $b=0$, such that the state equation $\calA y-\calB u=0$ reads $y'+\Lambda y=0$, $y(0)=u$. The cost functional is given by
	\begin{equation*}
		J(y)=\frac12\int_0^T\|y(t)-y_d(t)\|_H^2\,\mathrm{d}t+\frac12\|y(0)-y_{d,0}\|_H^2+\frac12\|y(T)-y_{d,T}\|_H^2
	\end{equation*}
	with $y_{d,0},y_{d,T}\in H$ and $y_d\in L^2(0,T;H)$, and $\theta=0$, i.e., we do not impose control constraints. By Theorem~\ref{thm:existence_optimal_sol}, there is a unique optimal pair $(\yopt,\uopt)$ and a unique multiplier $\popt=(\lopt,\lopt_0)\in L^2(0,T;V)\times H$ with $\calA'\popt=-DJ(\yopt)$ and $\uopt=\tfrac1\alpha\calB'\popt=\tfrac1\alpha\lopt_0$. A priori, however, we only know that $\lopt\in L^2(0,T;V)$, i.e., the adjoint equation only holds in $W([0,T])'$. Indeed, by definition of $\calA'$, solutions of the equation $\calA'(\lambda,\lambda_0)=\vartheta$ are pairs $(\lambda,\lambda_0)\in L^2(0,T;V)\times H$, which, for general right-hand sides $\vartheta\in W([0,T])'$, neither have initial or terminal values nor are continuous in time. In particular, the second component $\lambda_0$ is a priori not related to $\lambda$.

	\smallskip

	\noindent\textbf{Dual companion and balanced optimality system.} To balance the optimality system, we choose
	\begin{equation*}
		E:W([0,T])\to\calH\coloneqq L^2(0,T;V)\times H\times H,\qquad Ey\coloneqq(y,y(0),y(T)),
	\end{equation*}
	which is continuous with dense range, see \cite[Lem.~2.3]{schiela2013concise}, and obviously injective. Then $J=f\circ E$ with
	\begin{equation*}
		f(h,h_0,h_T)=\tfrac12\|h-y_d\|^2_{L^2(0,T;H)}+\tfrac12\|h_0-y_{d,0}\|_H^2+\tfrac12\|h_T-y_{d,T}\|_H^2,
	\end{equation*}
	which is Fr\'echet differentiable on $\calH$ as $L^2(0,T;V)\hookrightarrow L^2(0,T;H)$, and its derivative is given by
    \begin{equation*}
        Df(h,h_0,h_T)=(h-y_d,h_0-y_{d,0},h_T-y_{d,T})\in\calH'=L^2(0,T;V')\times H\times H,
    \end{equation*}
    where $h-y_d\in L^2(0,T;H)\subset L^2(0,T;V')$ by the Gelfand triplet.
	\begin{prpstn}\label{prop:parabolic_companion}
		The dual companion of $(W([0,T]),E,\calA)$ is given by $\Xd=W([0,T])\times H$, the inclusion $\Ed$ of $W([0,T])\times H$ into $\calP$ and
		\begin{equation*}
			\Ad:\Xd\to\calH',\qquad \Ad\stack{\lambda}{\lambda_0}=\begin{bsmallmatrix}-\lambda'+\Lambda'\lambda\\ \lambda_0-\lambda(0)\\ \lambda(T)\end{bsmallmatrix},
		\end{equation*}
		where $\Lambda':L^2(0,T;V)\to L^2(0,T;V')$ denotes the Banach space adjoint of $\Lambda$, i.e., $\langle\Lambda'\lambda,y\rangle=\langle\Lambda y,\lambda\rangle$.
	\end{prpstn}
	\begin{proof}
		We apply Proposition~\ref{prop:companion_unique} with $\tXd\coloneqq W([0,T])\times H$, the inclusion $\tEd$ into $\calP$ and $\tAd\coloneqq\Ad$ as defined above. By \eqref{eq:ibp_W}, we obtain for all $(\lambda,\lambda_0)\in\tXd$ and $y\in W([0,T])$
		\begin{align*}
			\left(\calA'\tEd\stack{\lambda}{\lambda_0}\right)(y)&=\langle y'+\Lambda y,\lambda\rangle+\langle y(0),\lambda_0\rangle_H\\
			&=\langle -\lambda'+\Lambda'\lambda,y\rangle+\langle \lambda_0-\lambda(0),y(0)\rangle_H+\langle \lambda(T),y(T)\rangle_H\\
			&=\left\langle\tAd\stack{\lambda}{\lambda_0},Ey\right\rangle_{\calH',\calH},
		\end{align*}
		i.e., the identity \eqref{eq:companion_identity_Y}, which reduces to \cite[Prop.~3.6]{schiela2013concise} for $\lambda_0=\lambda(0)$. Moreover, $\tAd$ is an isomorphism. Indeed, for $(g,g_0,g_T)\in\calH'$, the time reversal $q(t)\coloneqq \lambda(T-t)$ transforms the backward problem $-\lambda'+\Lambda'\lambda=g$, $\lambda(T)=g_T$ into a forward problem of the above form with $\Lambda$ replaced by $\mathcal R_T\Lambda'\mathcal R_T$, where $(\mathcal R_Tv)(t)\coloneqq v(T-t)$, which is again continuous and satisfies the same coercivity estimate. Hence, \cite[Thm.~3.4]{schiela2013concise} yields a unique solution $\lambda\in W([0,T])$ which depends continuously on the data, and $\lambda_0\coloneqq g_0+\lambda(0)$ is the unique choice with $\tAd(\lambda,\lambda_0)=(g,g_0,g_T)$. Consequently, Proposition~\ref{prop:companion_unique} yields the assertion.
	\end{proof}
	In particular, $\dom{\calA_E'}=W([0,T])\times H$ and $\calA_E'=\Ad$ by Proposition~\ref{prop:companion_unique}. Together with Lemma~\ref{lem:schiela_38_equiv}, we conclude for $(\lambda,\lambda_0)\in L^2(0,T;V)\times H$ that $\calA'(\lambda,\lambda_0)\in\ran{E'}$ holds if and only if $\lambda\in W([0,T])$. This is precisely the a posteriori lifting of the adjoint state in \cite[Prop.~3.8]{schiela2013concise}, i.e., although the multiplier is a priori only an element of $L^2(0,T;V)\times H$, the regularity of the cost lifts it into $W([0,T])\times H$.
	\begin{crllr}\label{cor:parabolic_optsys}
		The triple $(\yopt,\uopt,\popt)$ satisfies \eqref{eq:optsys} if and only if $\popt=(\lopt,\lopt_0)$ for some $(\lopt,\lopt_0)\in W([0,T])\times H$ with
		\begin{equation}\label{eq:optsys_parabolic}
	\left\{
	\begin{alignedat}{2}
		(\yopt)'+\Lambda\yopt&=0, &\qquad \yopt(0)&=\uopt,\\
		-\lopt'+\Lambda'\lopt&=-(\yopt-y_d), &\qquad \lopt(T)&=-(\yopt(T)-y_{d,T}),\\
		\uopt&=\tfrac1\alpha\lopt_0, &\qquad \lopt_0&=\lopt(0)-(\yopt(0)-y_{d,0}).
	\end{alignedat}
	\right.
\end{equation}
		In particular, \eqref{eq:optsys_parabolic} has a unique solution.
	\end{crllr}
	\begin{proof}
		As $\theta=0$, the proximal operator $\prox_{\theta/\alpha}$ is the identity, and $\calB'(\lopt,\lopt_0)=\lopt_0$, such that the third line of \eqref{eq:general_balanced_optsys_direct} reads $\uopt=\tfrac1\alpha\lopt_0$. Moreover, the state equation $\calA\yopt-\calB\uopt=0$ is equivalent to the first line of \eqref{eq:optsys_parabolic}, and $\Ad(\lopt,\lopt_0)=-Df(E\yopt)$ is equivalent to the second line together with $\lopt_0=\lopt(0)-(\yopt(0)-y_{d,0})$. Hence, the characterization follows from Theorem~\ref{thm:adjoint_regularity_general} and Proposition~\ref{prop:parabolic_companion}. Again, the optimality conditions in the proof of Theorem~\ref{thm:existence_optimal_sol} are necessary and sufficient, such that \eqref{eq:optsys} and hence \eqref{eq:optsys_parabolic} have exactly one solution.
	\end{proof}
	The system \eqref{eq:optsys_parabolic} is precisely the optimality system of \cite[Thm.~4.1]{schiela2013concise}. In this form, the regularity of state and adjoint is balanced, and they even enjoy the same regularity, i.e., $\yopt,\lopt\in W([0,T])$. Note that $\lopt_0$ differs from $\lopt(0)$ due to the observation of the initial value in the cost functional.
		\subsection{Abstract Cauchy problems}\label{subsec:cauchy}
	Many linear evolution equations, such as heat, wave, transport or Schr\"odinger equations, as well as linear ODEs, can be written as an abstract Cauchy problem
	\begin{equation}\label{eq:abstract_Cauchy}
		\tfrac{\mathrm{d}}{\mathrm{d}t}x(t)=Ax(t)+Bu(t),\qquad x(0)=x_0,\qquad t\in(0,T),
	\end{equation}
	where, throughout this subsection, $X$ and $U$ are Hilbert spaces, $A:X\supset\dom{A}\to X$ is the generator of a $C_0$-semigroup $(\mathbb{T}(t))_{t\geq0}$ on $X$, and $B\in L(U,X)$. Here, the spatial differential operator and the boundary conditions are encoded in the single operator $A$, see, e.g., \cite{curtain2020introduction,tucsnak2009observation}. While semigroup theory is a standard tool in mathematical systems theory, linear evolution equations in PDE-constrained optimization are usually treated by variational methods, i.e., by means of weak solutions. Both perspectives are, however, equivalent: by a well-known result of \cite{ball1977strongly}, weak solutions of \eqref{eq:abstract_Cauchy} coincide with the mild solutions given by the variation of constants formula. Hence, abstract Cauchy problems provide a unified framework for a broad class of linear ODEs and PDEs, which does not require analyticity, coercivity or maximal regularity, and thus also covers hyperbolic problems.

	We follow the strategy of Subsection~\ref{subsec:parabolic}, i.e., we write the state equation as an operator equation with an isomorphism $\calA$ and identify the dual companion by Proposition~\ref{prop:companion_unique}. In contrast to the variational setting of Subsection~\ref{subsec:parabolic}, however, a general generator $A$ does not enjoy maximal regularity, e.g., for hyperbolic problems, such that the isomorphism results for analytic semigroups, see \cite[Part~II, Ch.~1, Thm.~3.1]{bensoussan2007representation}, do not apply. Hence, we first have to establish the required isomorphism, and we proceed in three steps. First, we derive the natural solution space of \eqref{eq:abstract_Cauchy}, which consists of all continuous $X$-valued functions whose residual in the differential equation, understood in an extrapolation space, is integrable in time. Second, we show that the operator which maps a state to its residual and its initial value is an isomorphism, such that the results of Sections~\ref{sec:abstract_ocp} and~\ref{sec:balancing} apply. Third, we identify the dual companion, which shows that the multiplier is given by the solution $\lopt$ of an evolution equation posed backward in time and governed by the adjoint generator $A^*$, in a solution space of the same kind as the state.

	We denote the \textit{Hilbert space adjoint} of $B$ by $B^*\in L(X,U)$. The Hilbert space adjoint $A^*:X\supset\dom{A^*}\to X$ of the densely defined operator $A$ has the domain
	\[
	\dom{A^*}=\left\lbrace y\in X : \exists\, z\in X \text{ s.t.\ }\forall\,x\in\dom A:\;\langle y,Ax\rangle_{X}=\langle z,x\rangle_{X} \right\rbrace
	\]
	and is given by $A^*y=z$, which is well-defined by the density of $\dom A$ in $X$. Then $A^*$ generates the $C_0$-semigroup $(\mathbb{T}^*(t))_{t\geq0}$, see \cite[Thm.~2.3.6]{curtain2020introduction}.

	\smallskip

	\noindent\textbf{Interpolation and extrapolation spaces.} As $A$ is unbounded, the residual $x'-Ax$ of a mild solution of \eqref{eq:abstract_Cauchy} is, in general, not an element of $X$. To give it a meaning, we employ the scale of interpolation and extrapolation spaces associated with $A$, see \cite[Sec.~2.10]{tucsnak2009observation}. It is well known that the generator $A$ is densely defined and closed and has nonempty resolvent set $\rho(A)$, see, e.g., \cite[Thm.~2.1.13 and Lem.~2.1.14]{curtain2020introduction}. Fix $\mu\in\rho(A)$. The \textit{interpolation space associated with $A$} is defined by
	\begin{equation*}
		X_1\coloneqq\dom{A},\qquad \|z\|_{X_1}\coloneqq\|(\mu I-A)z\|_X,\qquad z\in\dom{A}.
	\end{equation*}
	Different choices of $\mu\in\rho(A)$ yield equivalent norms on $X_1$, such that $X_1$ is independent of the particular choice of $\mu$, up to equivalence of norms. Moreover, $\|\cdot\|_{X_1}$ is equivalent to the graph norm of $A$, and $X_1$ is a Hilbert space, see \cite[Prop.~2.10.1]{tucsnak2009observation}. In particular, its inner product may equivalently be chosen as the graph inner product
	\begin{equation*}
		\langle z,v\rangle_{\dom{A}}\coloneqq\langle z,v\rangle_X+\langle Az,Av\rangle_X,\qquad z,v\in\dom{A},
	\end{equation*}
	and the canonical embedding $X_1\hookrightarrow X$ is continuous and dense. Analogously, since $\overline{\mu}\in\rho(A^*)$, we define the \textit{interpolation space associated with $A^*$} by
	\begin{equation*}
		X_1^d\coloneqq\dom{A^*},\qquad \|z\|_{X_1^d}\coloneqq\|(\overline{\mu}I-A^*)z\|_X,\qquad z\in\dom{A^*}.
	\end{equation*}
	The \textit{extrapolation space associated with $A$} is obtained by equipping $X$ with the norm
	\begin{equation*}
		\|z\|_{X_{-1}}\coloneqq\|(\mu I-A)^{-1}z\|_X,\qquad z\in X,
	\end{equation*}
	and defining $X_{-1}$ as the completion of $X$ with respect to this norm. Again, different choices of $\mu\in\rho(A)$ lead to equivalent norms and hence to the same extrapolation space up to canonical isomorphism. The operator $A\in L(X_1,X)$ has a unique bounded extension $A_{-1}\in L(X,X_{-1})$, the \textit{extrapolation of $A$}, see \cite[Prop.~2.10.3]{tucsnak2009observation}, and we obtain the continuous and dense embeddings
	\begin{equation*}
		X_1\hookrightarrow X\hookrightarrow X_{-1}.
	\end{equation*}
	With the roles of $A$ and $A^*$ interchanged, we analogously obtain the extrapolation space $X_{-1}^d$ together with the extrapolation $(A^*)_{-1}\in L(X,X_{-1}^d)$ of $A^*$. In the Hilbert space setting, the extrapolation spaces can be identified with the duals of the interpolation spaces with respect to the pivot space $X$, i.e.,
	\begin{equation*}
		(X_1^d)'\simeq X_{-1},\qquad (X_1)'\simeq X_{-1}^d,
	\end{equation*}
	see \cite[Prop.~1.3.15]{Hosfeld2025}. Consequently,
	\begin{equation*}
		X_1^d\hookrightarrow X\simeq X'\hookrightarrow X_{-1}
	\end{equation*}
	forms a Gelfand triplet, also referred to as an \textit{evolution triple} or a \textit{rigged Hilbert space}, and the duality pairing between $X_{-1}$ and $X_1^d$ extends the inner product on the pivot space $X$, i.e.,
	\begin{equation*}
		\langle h,v\rangle_{X_{-1},X_1^d}=\langle h,v\rangle_X,\qquad h\in X,\ v\in X_1^d.
	\end{equation*}
	Moreover, the extrapolation of $A$ is characterized by
	\begin{equation}\label{eq:A_extrapolation}
		\langle A_{-1}z,v\rangle_{X_{-1},X_1^d}=\langle z,A^*v\rangle_X,\qquad z\in X,\ v\in X_1^d,
	\end{equation}
	see \cite[Prop.~2.10.3 and its proof]{tucsnak2009observation}. In particular, if $z\in X$ satisfies $A_{-1}z\in X$, then \eqref{eq:A_extrapolation} and $A^{**}=A$ imply $z\in\dom{A}$ and $Az=A_{-1}z$. The same arguments with $A$ and $A^*$ interchanged yield $\langle (A^*)_{-1}z,v\rangle_{X_{-1}^d,X_1}=\langle z,Av\rangle_X$ for all $z\in X$ and $v\in X_1$.

	\smallskip

	\noindent\textbf{Forward and backward operators.} We recall the two standard notions of solutions of the abstract Cauchy problem with a general source term.
	\begin{dfntn}[Solution concepts]\label{def:solution_concepts}
		Let $p\in[1,\infty)$, $x_0\in X$ and $g\in L^p(0,T;X)$.
		\begin{enumerate}[label=(\roman*)]
			\item The function
			\begin{equation*}
				x(t)=\mathbb{T}(t)x_0+\int_0^t\mathbb{T}(t-s)g(s)\,\mathrm{d}s,\qquad t\in[0,T],
			\end{equation*}
			which belongs to $C([0,T];X)$, is called the \textit{mild solution} of $x'=Ax+g$, $x(0)=x_0$.
			\item A function $x\in C([0,T];X)$ is called a \textit{weak solution} of $x'=Ax+g$, $x(0)=x_0$, if $x(0)=x_0$ and, for every $v\in\dom{A^*}$, the function $t\mapsto\langle x(t),v\rangle_X$ is absolutely continuous with
			\begin{equation}\label{eq:weak_sol}
				\tfrac{\mathrm{d}}{\mathrm{d}t}\langle x(t),v\rangle_X=\langle x(t),A^*v\rangle_X+\langle g(t),v\rangle_X\qquad\text{for a.e. }t\in(0,T).
			\end{equation}
		\end{enumerate}
	\end{dfntn}
	By \cite{ball1977strongly}, see also \cite[Thm.~5.1.10]{curtain2020introduction}, for every $x_0\in X$ and $g\in L^p(0,T;X)$, there exists a unique weak solution, which coincides with the mild solution.

	For a Banach space $Y$, we denote by $D'(0,T;Y)$ the space of $Y$-valued distributions on $(0,T)$, i.e., of continuous linear mappings from $C_c^\infty(0,T)$ to $Y$. Every $x\in L^1(0,T;Y)$, and in particular every $x\in C([0,T];Y)$, defines an element of $D'(0,T;Y)$ via $\varphi\mapsto\int_0^T\varphi(t)x(t)\,\mathrm{d}t$, and its distributional derivative $x'\in D'(0,T;Y)$ is given by
	\begin{equation*}
		x'(\varphi)\coloneqq-\int_0^T\varphi'(t)\,x(t)\,\mathrm{d}t,\qquad\varphi\in C_c^\infty(0,T).
	\end{equation*}
	As $X\hookrightarrow X_{-1}$ and $A_{-1}\in L(X,X_{-1})$, we may regard $x'$ and $A_{-1}x$ as elements of $D'(0,T;X_{-1})$ for every $x\in C([0,T];X)$. 	For $p\in(1,\infty)$, we define the solution space
	\begin{equation*}
		E_p(0,T;A)\coloneqq\left\{x\in C([0,T];X): x'-A_{-1}x\in L^p(0,T;X)\right\},
	\end{equation*}
	endowed with the norm
	\begin{equation*}
		\|x\|_{E_p(0,T;A)}\coloneqq\|x\|_{C([0,T];X)}+\|x'-A_{-1}x\|_{L^p(0,T;X)}.
	\end{equation*}
	Further, let $q\coloneqq\tfrac{p}{p-1}$ and define the reflexive Banach space $\calP\coloneqq L^q(0,T;X)\times X$, whose dual space we identify with
	\begin{equation*}
		\calP'=L^p(0,T;X)\times X
	\end{equation*}
	via the pairing
	\begin{equation*}
		\left\langle\stack{g}{x_0},\stack{\lambda}{\lambda_0}\right\rangle_{\calP',\calP}\coloneqq\int_0^T\langle g(t),\lambda(t)\rangle_X\,\mathrm{d}t+\langle x_0,\lambda_0\rangle_X,
	\end{equation*}
	which is possible as $X$ is a Hilbert space. We consider the linear \emph{forward operator} $\calA:E_p(0,T;A)\to\calP'$ given by
	\begin{equation}\label{eq:def_A}
		\calA x\coloneqq\stack{x'-A_{-1}x}{x(0)}.
	\end{equation}
	\begin{lmm}\label{lem:Ep_Banach}
		The space $E_p(0,T;A)$ is a Banach space, and $\calA:E_p(0,T;A)\to \calP'$ is a bounded linear operator.
	\end{lmm}
	\begin{proof}
		We show that the linear operator $\mathcal U:C([0,T];X)\supset E_p(0,T;A)\to L^p(0,T;X)$, $\mathcal Ux\coloneqq x'-A_{-1}x$, is closed. To this end, let $(x_n)_{n\in\N}\subset E_p(0,T;A)$ be such that
		\begin{equation*}
			x_n\to x\quad\text{in }C([0,T];X),\qquad \mathcal Ux_n\to g\quad\text{in }L^p(0,T;X),
		\end{equation*}
		as $n\to\infty$. We show that $x\in E_p(0,T;A)$ and $\mathcal Ux=g$. First, for every $\varphi\in C_c^\infty(0,T)$, we have
		\begin{equation*}
			x_n'(\varphi)=-\int_0^T\varphi'(t)\,x_n(t)\,\mathrm{d}t\to-\int_0^T\varphi'(t)\,x(t)\,\mathrm{d}t=x'(\varphi)\quad\text{in }X\hookrightarrow X_{-1},
		\end{equation*}
		i.e., $x_n'\to x'$ in $D'(0,T;X_{-1})$. Because $A_{-1}\in L(X,X_{-1})$, we obtain $A_{-1}x_n\to A_{-1}x$ in $C([0,T];X_{-1})$ and hence in $D'(0,T;X_{-1})$. Consequently, $x_n'-A_{-1}x_n\to x'-A_{-1}x$ in $D'(0,T;X_{-1})$. On the other hand, since $L^p(0,T;X)\hookrightarrow D'(0,T;X_{-1})$ continuously, we have $x_n'-A_{-1}x_n\to g$ in $D'(0,T;X_{-1})$. By uniqueness of limits, as $D'(0,T;X_{-1})$ is a Hausdorff space, we obtain
		\begin{equation*}
			x'-A_{-1}x=g\in L^p(0,T;X).
		\end{equation*}
		This implies $x\in E_p(0,T;A)$ and $\mathcal Ux=g$. Hence, $\mathcal U$ is closed, such that its domain endowed with the graph norm is a Banach space. As the graph norm of $\mathcal U$ coincides with $\|\cdot\|_{E_p(0,T;A)}$, the space $E_p(0,T;A)$ is a Banach space. Finally, $\calA$ is bounded by definition of the norm on $E_p(0,T;A)$ and the continuity of the evaluation $C([0,T];X)\ni x\mapsto x(0)\in X$.
	\end{proof}
	The following result shows that the operator $\calA$ provides an equivalent formulation of the abstract Cauchy problem, which links the variational and the semigroup perspective. 	For generators of analytic semigroups, maximal regularity yields that $x\mapsto(x'-Ax,x(0))$ is an isomorphism from $W^{1,2}(0,T;X)\cap L^2(0,T;\dom{A})$ onto $L^2(0,T;X)\times D_A(\tfrac12,2)$, with an interpolation space $D_A(\tfrac12,2)$ of initial values, see \cite[Part~II, Ch.~1, Thm.~3.1]{bensoussan2007representation}. For coercive variational problems, an analogous isomorphism holds in the setting of Subsection~\ref{subsec:parabolic}, see \cite[Part~II, Ch.~2, Thm.~1.1]{bensoussan2007representation}. Theorem~\ref{thm:inv_of_A} extends these results to generators of arbitrary $C_0$-semigroups and initial values in $X$, where maximal regularity is replaced by the requirement that the residual $x'-A_{-1}x$, understood in the extrapolation space $X_{-1}$, is $p$-integrable.
	\begin{thrm}[Isomorphism property of the forward operator]\label{thm:inv_of_A}
		Let $p\in (1,\infty)$, $x_0\in X$ and $g\in L^p(0,T;X)$. Then $x\in E_p(0,T;A)$ satisfies
		\begin{equation*}
			\calA x=\stack{g}{x_0}
		\end{equation*}
		if and only if $x$ is a weak solution of $x'=Ax+g$, $x(0)=x_0$. In particular, $\calA$ is an isomorphism from $E_p(0,T;A)$ onto $\calP'$, and
		\begin{equation}\label{eq:mild_sol_general}
			\calA^{-1}\stack{g}{x_0}=\mathbb{T}(\cdot)x_0+\int_0^\cdot\mathbb{T}(\cdot-s)g(s)\,\mathrm{d}s.
		\end{equation}
	\end{thrm}
	\begin{proof}
		First, let $x\in C([0,T];X)$ be a weak solution of $x'=Ax+g$, $x(0)=x_0$. For all $v\in X_1^d=\dom{A^*}$ and $\varphi\in C_c^\infty(0,T)$, the absolute continuity of $t\mapsto\langle x(t),v\rangle_X$, \eqref{eq:weak_sol} and \eqref{eq:A_extrapolation} yield
		\begin{align*}
			\int_0^T\varphi'(t)\,\langle x(t),v\rangle_{X_{-1},X_1^d}\,\mathrm{d}t&=\int_0^T\varphi'(t)\,\langle x(t),v\rangle_X\,\mathrm{d}t
			=-\int_0^T\varphi(t)\,\tfrac{\mathrm{d}}{\mathrm{d}t}\langle x(t),v\rangle_X\,\mathrm{d}t\\
			&=-\int_0^T\varphi(t)\big(\langle x(t),A^*v\rangle_X+\langle g(t),v\rangle_X\big)\,\mathrm{d}t\\
			&=-\int_0^T\varphi(t)\,\langle A_{-1}x(t)+g(t),v\rangle_{X_{-1},X_1^d}\,\mathrm{d}t.
		\end{align*}
		As $X_{-1}\simeq(X_1^d)'$, the elements of $X_{-1}$ are separated by $X_1^d$, and we conclude
		\begin{equation*}
			x'-A_{-1}x=g\qquad\text{in }D'(0,T;X_{-1}).
		\end{equation*}
		Since $g\in L^p(0,T;X)$, this shows $x\in E_p(0,T;A)$. Together with $x(0)=x_0$, we obtain $\calA x=(g,x_0)$.

		Conversely, let $x\in E_p(0,T;A)$ satisfy $\calA x=(g,x_0)$. Then $x(0)=x_0$, and the above computation shows that \eqref{eq:weak_sol} holds in the sense of distributions on $(0,T)$ for all $v\in X_1^d$. It remains to show that $t\mapsto\langle x(t),v\rangle_X$ is absolutely continuous for every $v\in X_1^d$. By definition of $\calA$, we have $x'-A_{-1}x=g$ in $D'(0,T;X_{-1})$. Since $x\in C([0,T];X)$ and $A_{-1}\in L(X,X_{-1})$, we have
		\begin{equation*}
			A_{-1}x\in C([0,T];X_{-1})\hookrightarrow L^p(0,T;X_{-1}).
		\end{equation*}
		Moreover, $g\in L^p(0,T;X)\hookrightarrow L^p(0,T;X_{-1})$, and consequently
		\begin{equation*}
			x'=A_{-1}x+g\in L^p(0,T;X_{-1}),
		\end{equation*}
		i.e., $x\in W^{1,p}(0,T;X_{-1})$. Let $v\in X_1^d$. Since $(X_{-1})'\simeq X_1^d$, the mapping $X_{-1}\ni z\mapsto\langle z,v\rangle_{X_{-1},X_1^d}$ is a bounded linear functional. As bounded linear operators preserve Bochner--Sobolev regularity, the function $t\mapsto\langle x(t),v\rangle_{X_{-1},X_1^d}$ belongs to $W^{1,p}(0,T)$ with
		\begin{equation*}
			\tfrac{\mathrm{d}}{\mathrm{d}t}\langle x(t),v\rangle_{X_{-1},X_1^d}=\langle x'(t),v\rangle_{X_{-1},X_1^d}\qquad\text{for a.e. }t\in(0,T).
		\end{equation*}
		In particular, this function is absolutely continuous. Since $x(t)\in X$ for every $t\in[0,T]$ and the duality pairing extends the inner product on $X$, we have $\langle x(t),v\rangle_{X_{-1},X_1^d}=\langle x(t),v\rangle_X$. Consequently, $t\mapsto\langle x(t),v\rangle_X$ is absolutely continuous, and $x$ is a weak solution.

		By \cite[Thm.~5.1.10]{curtain2020introduction}, for every $(g,x_0)\in L^p(0,T;X)\times X$, there exists a unique weak solution, which is given by the mild solution. By the equivalence shown above, $\calA$ is bijective and \eqref{eq:mild_sol_general} holds. Finally, $\calA$ is bounded by Lemma~\ref{lem:Ep_Banach}, and its inverse is bounded by the open mapping theorem.
	\end{proof}
	From now on, we restrict ourselves to $p=q=2$, such that $\calP=\calP'=L^2(0,T;X)\times X$ is a Hilbert space, the above pairing coincides with its inner product, and $\calA:E_2(0,T;A)\to\calP'$ is an isomorphism.

	Analogously, we define the solution space of the backward problem
	\begin{equation*}
		E_2(0,T;-A^*)\coloneqq\left\{\lambda\in C([0,T];X):\lambda'+(A^*)_{-1}\lambda\in L^2(0,T;X)\right\},
	\end{equation*}
	endowed with the norm $\|\lambda\|_{E_2(0,T;-A^*)}\coloneqq\|\lambda\|_{C([0,T];X)}+\|\lambda'+(A^*)_{-1}\lambda\|_{L^2(0,T;X)}$, where $\lambda'+(A^*)_{-1}\lambda$ is understood in $D'(0,T;X_{-1}^d)$, and the \emph{backward operator}
	\begin{equation}\label{eq:bkw_cauchy}
		\Ad:E_2(0,T;-A^*)\to\calP,\qquad \Ad \lambda\coloneqq\stack{-\lambda'-(A^*)_{-1}\lambda}{\lambda(T)}.
	\end{equation}
	\begin{lmm}\label{cor:backward_iso}
		The operator $\Ad$ is an isomorphism and
		\begin{equation*}
			(\Ad)^{-1}\stack{g}{\lambda_T}=\mathbb{T}^*(T-\cdot)\lambda_T+\int_\cdot^T\mathbb{T}^*(s-\cdot)g(s)\,\mathrm{d}s.
		\end{equation*}
	\end{lmm}
	\begin{proof}
		Denote by $E_2(0,T;A^*)$ and $\calA_{A^*}$ the space and the operator of Theorem~\ref{thm:inv_of_A} with $A$ replaced by $A^*$, such that $\calA_{A^*}$ is an isomorphism. For the time reversal $(\mathcal R_T\lambda)(t)\coloneqq \lambda(T-t)$, extended to distributions via $(\mathcal R_Tw)(\varphi)\coloneqq w(\varphi(T-\cdot))$, the substitution $s=T-t$ yields $(\mathcal R_T\lambda)'=-\mathcal R_T\lambda'$ and hence
		\begin{equation*}
			(\mathcal R_T\lambda)'-(A^*)_{-1}\mathcal R_T\lambda=-\mathcal R_T\big(\lambda'+(A^*)_{-1}\lambda\big).
		\end{equation*}
		Thus, $\mathcal R_T$ is an isometric isomorphism from $E_2(0,T;-A^*)$ onto $E_2(0,T;A^*)$, and $\Ad=\diag(\mathcal R_T,I)\,\calA_{A^*}\mathcal R_T$ is an isomorphism. The formula follows from \eqref{eq:mild_sol_general} applied to $A^*$ by substitution.
	\end{proof}

	\smallskip

	\noindent\textbf{Dual companion of the forward operator.} The link between the forward and the backward operator is given by the following integration by parts formula.
	\begin{prpstn}[Green identity]\label{prop:Green}
		Let $h\in E_2(0,T;A)$ and $\lambda\in E_2(0,T;-A^*)$, and define
		\begin{equation*}
			f_h\coloneqq h'-A_{-1}h\in L^2(0,T;X),\qquad g_\lambda\coloneqq \lambda'+(A^*)_{-1}\lambda\in L^2(0,T;X).
		\end{equation*}
		Then the function $t\mapsto\langle \lambda(t),h(t)\rangle_X$ belongs to $W^{1,1}(0,T)$, and for all $0\le s\le t\le T$,
		\begin{equation}\label{eq:int_by_parts}
			\langle \lambda(t),h(t)\rangle_X-\langle \lambda(s),h(s)\rangle_X=\int_s^t\langle g_\lambda(\tau),h(\tau)\rangle_X+\langle \lambda(\tau),f_h(\tau)\rangle_X\,\mathrm{d}\tau.
		\end{equation}
	\end{prpstn}
	\begin{proof} Let $x_0\in\dom{A}$, $g\in C^1([0,T];X)$ and $x\coloneqq\calA^{-1}(g,x_0)$ as in \eqref{eq:mild_sol_general}. By \cite[Thm. 5.1.3]{curtain2020introduction}, $x\in C^1([0,T];X)$, and $A_{-1}x=x'-g\in C([0,T];X)$ implies $x\in C([0,T];X_1)$ by the consequence of \eqref{eq:A_extrapolation} stated above. As $\dom{A}\times C^1([0,T];X)$ is dense in $X\times L^2(0,T;X)$ and $\calA^{-1}$ is continuous, the set $C^1([0,T];X)\cap C([0,T];X_1)$ is dense in $E_2(0,T;A)$. Applying this to $A^*$ and using the time reversal of the proof of Lemma~\ref{cor:backward_iso}, which leaves $C^1([0,T];X)\cap C([0,T];X_1^d)$ invariant, the latter set is dense in $E_2(0,T;-A^*)$.

		Let $h\in C^1([0,T];X)\cap C([0,T];X_1)$ and $\lambda\in C^1([0,T];X)\cap C([0,T];X_1^d)$. Then
		\begin{equation*}
			f_h=h'-A_{-1}h=h'-Ah\in C([0,T];X),\qquad g_\lambda=\lambda'+(A^*)_{-1}\lambda=\lambda'+A^*\lambda\in C([0,T];X).
		\end{equation*}
		By the product rule and as $\langle \lambda(t),Ah(t)\rangle_X=\langle A^*\lambda(t),h(t)\rangle_X$, we obtain
		\begin{align*}
			\tfrac{\mathrm{d}}{\mathrm{d}t}\langle \lambda(t),h(t)\rangle_X&=\langle \lambda'(t),h(t)\rangle_X+\langle \lambda(t),h'(t)\rangle_X\\
			&=\langle \lambda'(t)+A^*\lambda(t),h(t)\rangle_X+\langle \lambda(t),h'(t)-Ah(t)\rangle_X\\
			&=\langle g_\lambda(t),h(t)\rangle_X+\langle \lambda(t),f_h(t)\rangle_X.
		\end{align*}
		Integrating from $s$ to $t$ yields \eqref{eq:int_by_parts} for $h\in C^1([0,T];X)\cap C([0,T];X_1)$ and $\lambda\in C^1([0,T];X)\cap C([0,T];X_1^d)$.

		By the beginning of this proof, there are sequences $(h_n)_{n\in\N}\subset C^1([0,T];X)\cap C([0,T];X_1)$ and $(\lambda_n)_{n\in\N}\subset C^1([0,T];X)\cap C([0,T];X_1^d)$ such that
		\begin{equation*}
			h_n\to h\quad\text{in }E_2(0,T;A),\qquad \lambda_n\to \lambda\quad\text{in }E_2(0,T;-A^*),
		\end{equation*}
		as $n\to\infty$. Clearly, the identity \eqref{eq:int_by_parts} holds for $(h_n,\lambda_n)$ with $f_{h_n}=h_n'-Ah_n$ and $g_{\lambda_n}=\lambda_n'+A^*\lambda_n$. Since $h_n\to h$ and $\lambda_n\to \lambda$ in $C([0,T];X)$, the boundary terms converge, i.e.,
		\begin{equation*}
			\langle \lambda_n(t),h_n(t)\rangle_X-\langle \lambda_n(s),h_n(s)\rangle_X\to\langle \lambda(t),h(t)\rangle_X-\langle \lambda(s),h(s)\rangle_X,
		\end{equation*}
		as $n\to\infty$. Moreover, $f_{h_n}\to f_h$ and $g_{\lambda_n}\to g_\lambda$ in $L^2(0,T;X)$. We estimate
		\begin{align*}
			&\left|\int_s^t\Big(\langle g_{\lambda_n}(\tau),h_n(\tau)\rangle_X-\langle g_\lambda(\tau),h(\tau)\rangle_X\Big)\,\mathrm{d}\tau\right|\\
			&\hspace{2cm}=\left|\int_s^t\Big(\langle g_{\lambda_n}(\tau)-g_\lambda(\tau),h_n(\tau)\rangle_X+\langle g_\lambda(\tau),h_n(\tau)-h(\tau)\rangle_X\Big)\,\mathrm{d}\tau\right|\\
			&\hspace{2cm}\le\|g_{\lambda_n}-g_\lambda\|_{L^2(0,T;X)}\|h_n\|_{L^2(0,T;X)}+\|g_\lambda\|_{L^2(0,T;X)}\|h_n-h\|_{L^2(0,T;X)},
		\end{align*}
		which tends to $0$ as $n\to\infty$, as $(h_n)_{n\in\N}$ is bounded in $L^2(0,T;X)$. Analogously, $\int_s^t\langle \lambda_n,f_{h_n}\rangle_X\,\mathrm{d}\tau$ converges to $\int_s^t\langle \lambda,f_h\rangle_X\,\mathrm{d}\tau$. Passing to the limit yields \eqref{eq:int_by_parts}. As the integrand on the right-hand side of \eqref{eq:int_by_parts} belongs to $L^1(0,T)$, the function $t\mapsto\langle \lambda(t),h(t)\rangle_X$ belongs to $W^{1,1}(0,T)$.
	\end{proof}

	To balance the optimality system, we define the trace embeddings
	\begin{align*}
		E_{T,A}&:E_2(0,T;A)\to\calP, & E_{T,A}h&\coloneqq(h,h(T)),\\
		E_{0,-A^*}&:E_2(0,T;-A^*)\to\calP, & E_{0,-A^*}\lambda&\coloneqq(\lambda,\lambda(0)).
	\end{align*}
	\begin{lmm}\label{cor:ETA}
		The maps $E_{T,A}$ and $E_{0,-A^*}$ are continuous and injective with dense range. In particular, the adjoint $E_{T,A}'\in L(\calP,E_2(0,T;A)')$, given by
		\begin{equation*}
			(E_{T,A}'(f,x_T))(x)=\int_0^T\langle f(t),x(t)\rangle_X\,\mathrm{d}t+\langle x_T,x(T)\rangle_X,\qquad x\in E_2(0,T;A),
		\end{equation*}
		is injective.
	\end{lmm}
	\begin{proof}
		Let $h\in E_2(0,T;A)$. Since $\|h\|_{L^2(0,T;X)}\le T^{1/2}\|h\|_{C([0,T];X)}$ and $\|h(T)\|_X\le\|h\|_{C([0,T];X)}$, we obtain
		\begin{equation*}
			\|E_{T,A}h\|_\calP\le(T^{1/2}+1)\|h\|_{C([0,T];X)}\le(T^{1/2}+1)\|h\|_{E_2(0,T;A)},
		\end{equation*}
		i.e., $E_{T,A}$ is continuous. Clearly, $E_{T,A}$ is injective. We next prove that $E_{T,A}$ has dense range. Note that $C^1([0,T];X_1)\subset E_2(0,T;A)$, as $h'-A_{-1}h=h'-Ah\in C([0,T];X)$ for $h\in C^1([0,T];X_1)$. Let $(g,\xi)\in\calP$. As $X_1$ is dense in $X$, we may choose $(g_n)_{n\in\N}\subset C_c^1((0,T);X_1)$ and $(\xi_n)_{n\in\N}\subset X_1$ with
		\begin{equation*}
			g_n\to g\quad\text{in }L^2(0,T;X),\qquad \xi_n\to\xi\quad\text{in }X,
		\end{equation*}
		as $n\to\infty$. Further, let $\eta_n\in C^1([0,T])$ with $0\le\eta_n\le1$, $\eta_n(T)=1$ and $\operatorname{supp}\eta_n\subset[T-\tfrac1n,T]$, such that $\|\eta_n\xi_n\|_{L^2(0,T;X)}\le n^{-1/2}\|\xi_n\|_X\to0$. Then $h_n\coloneqq g_n+\eta_n\xi_n\in C^1([0,T];X_1)\subset E_2(0,T;A)$ satisfies $h_n(T)=\xi_n\to\xi$ in $X$ and
		\begin{equation*}
			\|h_n-g\|_{L^2(0,T;X)}\le\|g_n-g\|_{L^2(0,T;X)}+\|\eta_n\xi_n\|_{L^2(0,T;X)}\to0,
		\end{equation*}
		such that $E_{T,A}h_n=(h_n,h_n(T))\to(g,\xi)$ in $\calP$. This proves that $E_{T,A}$ has dense range, and $E_{T,A}'$ is injective as $\ker E_{T,A}'=\ran{E_{T,A}}^\perp=\{0\}$. The assertions for $E_{0,-A^*}$ follow analogously, using $C^1([0,T];X_1^d)\subset E_2(0,T;-A^*)$ and functions $\eta_n$ with $\eta_n(0)=1$ and $\operatorname{supp}\eta_n\subset[0,\tfrac1n]$.
	\end{proof}

	The Green identity yields the following relation between the forward operator, the backward operator and the trace embeddings.
	\begin{prpstn}\label{prop:cauchy_companion_identity}
		For all $\lambda\in E_2(0,T;-A^*)$ and $h\in E_2(0,T;A)$, we have
		\begin{equation*}
			(\calA'E_{0,-A^*}\lambda)(h)=(E_{T,A}'\Ad \lambda)(h),
		\end{equation*}
		i.e., $\calA'E_{0,-A^*}=E_{T,A}'\Ad:E_2(0,T;-A^*)\to E_2(0,T;A)'$.
	\end{prpstn}
	\begin{proof}
		By definition of the operators involved, we have
		\begin{align*}
			&(\calA'E_{0,-A^*}\lambda)(h)-(E_{T,A}'\Ad \lambda)(h)\\
            &=\int_0^T\langle h'(t)-A_{-1}h(t),\lambda(t)\rangle_X-\langle -\lambda'(t)-(A^*)_{-1}\lambda(t),h(t)\rangle_X\,\mathrm{d}t
            +\langle h(0),\lambda(0)\rangle_X-\langle \lambda(T),h(T)\rangle_X\\
            &=0,
		\end{align*}
		where the last equality follows from \eqref{eq:int_by_parts} with $s=0$ and $t=T$.
	\end{proof}
	\begin{thrm}[Dual companion of the forward operator]\label{thm:cauchy_companion}
		Let $\calA_E$ be the unbounded operator of Lemma~\Rref{lem:schiela_38_equiv} with $\calH=\calP$ and $E=E_{T,A}$, i.e., $\calA_E(h,h(T))\coloneqq\calA h$ with domain $E_{T,A}(E_2(0,T;A))\subset\calP$. Then
		\begin{gather*}
			\dom{\calA_E'}=E_{0,-A^*}(E_2(0,T;-A^*)),\\
			\calA_E'\stack{\lambda}{\lambda(0)}=\Ad \lambda=\stack{-\lambda'-(A^*)_{-1}\lambda}{\lambda(T)}\quad\text{for } \lambda\in E_2(0,T;-A^*),
		\end{gather*}
		and $E_{0,-A^*}$ is an isomorphism from $E_2(0,T;-A^*)$ onto $\dom{\calA_E'}$ endowed with the graph norm. In particular, the dual companion of $(E_2(0,T;A),E_{T,A},\calA)$ is represented by $(E_2(0,T;-A^*),\allowbreak E_{0,-A^*},\allowbreak\Ad)$.
	\end{thrm}
	\begin{proof}
		As $\calP$ is identified with its dual space, the definition of the Banach space adjoint yields that $(\lambda,\lambda_0)\in\calP$ belongs to $\dom{\calA_E'}$ if and only if there exists $(z,z_T)\in\calP$ such that
		\begin{equation}\label{eq:adjoint_inner_product}
			\int_0^T\langle \lambda(t),h'(t)-A_{-1}h(t)\rangle_X\,\mathrm{d}t+\langle \lambda_0,h(0)\rangle_X=\int_0^T\langle z(t),h(t)\rangle_X\,\mathrm{d}t+\langle z_T,h(T)\rangle_X
		\end{equation}
		for all $h\in E_2(0,T;A)$, and in this case, $\calA_E'(\lambda,\lambda_0)=(z,z_T)$.

		We first show $\dom{\calA_E'}\subset E_{0,-A^*}(E_2(0,T;-A^*))$. Let $(\lambda,\lambda_0)\in\dom{\calA_E'}$ and $(z,z_T)\coloneqq\calA_E'(\lambda,\lambda_0)$. For $\varphi\in C_c^\infty(0,T)$ and $v\in X_1$, the function $h\coloneqq\varphi v$ belongs to $C_c^1((0,T);X_1)\subset E_2(0,T;A)$ with $h(0)=h(T)=0$ and $A_{-1}h=\varphi Av$. As $\langle \lambda(t),Av\rangle_X=\langle(A^*)_{-1}\lambda(t),v\rangle_{X_{-1}^d,X_1}$, the identity \eqref{eq:adjoint_inner_product} yields
		\begin{equation*}
			\int_0^T\varphi'(t)\,\langle \lambda(t),v\rangle_{X_{-1}^d,X_1}\,\mathrm{d}t=\int_0^T\varphi(t)\,\langle z(t)+(A^*)_{-1}\lambda(t),v\rangle_{X_{-1}^d,X_1}\,\mathrm{d}t
		\end{equation*}
		for all $\varphi\in C_c^\infty(0,T)$ and $v\in X_1$. As $X_{-1}^d\simeq(X_1)'$, we conclude
		\begin{equation*}
			\lambda'=-z-(A^*)_{-1}\lambda\qquad\text{in }D'(0,T;X_{-1}^d).
		\end{equation*}
		Since $z\in L^2(0,T;X)$ and $\lambda\in L^2(0,T;X)$, we obtain $\lambda'\in L^2(0,T;X_{-1}^d)$, i.e., $\lambda\in W^{1,2}(0,T;X_{-1}^d)\hookrightarrow C([0,T];X_{-1}^d)$. In particular, the traces $\lambda(0),\lambda(T)\in X_{-1}^d$ are well-defined. For $h\in C^1([0,T];X_1)$, the integration by parts formula for $X_{-1}^d$-valued Sobolev functions yields
		\begin{equation*}
			\int_0^T\langle \lambda(t),h'(t)\rangle_{X_{-1}^d,X_1}\,\mathrm{d}t=\langle \lambda(T),h(T)\rangle_{X_{-1}^d,X_1}-\langle \lambda(0),h(0)\rangle_{X_{-1}^d,X_1}-\int_0^T\langle \lambda'(t),h(t)\rangle_{X_{-1}^d,X_1}\,\mathrm{d}t.
		\end{equation*}
		Using $\langle \lambda(t),Ah(t)\rangle_X=\langle(A^*)_{-1}\lambda(t),h(t)\rangle_{X_{-1}^d,X_1}$ and $\lambda'+(A^*)_{-1}\lambda=-z$, this yields
		\begin{align*}
			\int_0^T\langle \lambda(t),h'(t)-Ah(t)\rangle_X\,\mathrm{d}t&=\langle \lambda(T),h(T)\rangle_{X_{-1}^d,X_1}-\langle \lambda(0),h(0)\rangle_{X_{-1}^d,X_1}\\
			&\qquad-\int_0^T\langle \lambda'(t)+(A^*)_{-1}\lambda(t),h(t)\rangle_{X_{-1}^d,X_1}\,\mathrm{d}t\\
			&=\langle \lambda(T),h(T)\rangle_{X_{-1}^d,X_1}-\langle \lambda(0),h(0)\rangle_{X_{-1}^d,X_1}+\int_0^T\langle z(t),h(t)\rangle_X\,\mathrm{d}t.
		\end{align*}
		On the other hand, \eqref{eq:adjoint_inner_product} implies
		\begin{equation*}
			\int_0^T\langle \lambda(t),h'(t)-Ah(t)\rangle_X\,\mathrm{d}t=\int_0^T\langle z(t),h(t)\rangle_X\,\mathrm{d}t+\langle z_T,h(T)\rangle_X-\langle \lambda_0,h(0)\rangle_X.
		\end{equation*}
		Combining these two identities gives
		\begin{equation*}
			\langle \lambda(T)-z_T,h(T)\rangle_{X_{-1}^d,X_1}+\langle \lambda_0-\lambda(0),h(0)\rangle_{X_{-1}^d,X_1}=0\qquad\text{for all }h\in C^1([0,T];X_1).
		\end{equation*}
		Choosing $h\in C^1([0,T];X_1)$ with $h(0)=0$ and arbitrary $h(T)\in X_1$, and conversely with $h(T)=0$ and arbitrary $h(0)\in X_1$, yields $\lambda(T)=z_T$ and $\lambda(0)=\lambda_0$ in $X_{-1}^d\simeq(X_1)'$. As $z_T,\lambda_0\in X$, we conclude that $\lambda(T),\lambda(0)\in X$ and that these equalities hold in $X$.

		It remains to show $\lambda\in C([0,T];X)$. The time-reversed function $q\coloneqq \lambda(T-\cdot)$ satisfies $q\in W^{1,2}(0,T;X_{-1}^d)\cap L^2(0,T;X)$ and
		\begin{equation*}
			q'=(A^*)_{-1}q+z(T-\cdot)\quad\text{in }L^2(0,T;X_{-1}^d),\qquad q(0)=z_T.
		\end{equation*}
		As $(A^*)_{-1}$ with domain $X$ generates the extrapolation $(\mathbb{T}^*(t)_{-1})_{t\geq0}$ of the semigroup $(\mathbb{T}^*(t))_{t\geq0}$ to $X_{-1}^d$, see \cite[Sec.~2.10]{tucsnak2009observation}, $q$ is a solution of this problem in $X_{-1}^d$ in the sense of \cite[Def.~4.1.1]{tucsnak2009observation} with $A^*$ in place of $A$, and hence it is given by the variation of constants formula, see \cite[Prop.~4.1.4]{tucsnak2009observation}. Reversing time again, we obtain
		\begin{equation*}
			\lambda(t)=\mathbb{T}^*(T-t)_{-1}z_T+\int_t^T\mathbb{T}^*(s-t)_{-1}z(s)\,\mathrm{d}s,\qquad t\in[0,T].
		\end{equation*}
		Since $\mathbb{T}^*(t)_{-1}$ coincides with $\mathbb{T}^*(t)$ on $X$, and $z_T\in X$ as well as $z\in L^2(0,T;X)$, we conclude $\lambda\in C([0,T];X)$. Together with $\lambda'+(A^*)_{-1}\lambda=-z\in L^2(0,T;X)$, this shows $\lambda\in E_2(0,T;-A^*)$, $(\lambda,\lambda_0)=E_{0,-A^*}\lambda$ and $\calA_E'(\lambda,\lambda(0))=(z,z_T)=\Ad \lambda$.

		It remains to prove $E_{0,-A^*}(E_2(0,T;-A^*))\subset\dom{\calA_E'}$. Let $\lambda\in E_2(0,T;-A^*)$. By Proposition~\ref{prop:cauchy_companion_identity}, we have
		\begin{equation*}
			\left\langle\calA_E\stack{h}{h(T)},\stack{\lambda}{\lambda(0)}\right\rangle_\calP=(\calA'E_{0,-A^*}\lambda)(h)=(E_{T,A}'\Ad \lambda)(h)=\left\langle\Ad \lambda,\stack{h}{h(T)}\right\rangle_\calP
		\end{equation*}
		for all $h\in E_2(0,T;A)$. Hence, \eqref{eq:adjoint_inner_product} holds with $(\lambda_0,z,z_T)=(\lambda(0),-\lambda'-(A^*)_{-1}\lambda,\lambda(T))$, i.e., $(\lambda,\lambda(0))\in\dom{\calA_E'}$ and $\calA_E'(\lambda,\lambda(0))=\Ad \lambda$.

		A combination of Lemma~\ref{cor:ETA} and Lemma~\ref{cor:backward_iso} shows that
        \begin{equation*}
            \|\lambda\|_{E_2(0,T;-A^*)} \leq \|(\Ad)^{-1}\|\,\|\Ad \lambda\|_\calP\leq C \left( \|E_{0,-A^*}\lambda\|_{\calP} +\|\Ad \lambda\|_\calP  \right) \leq C' \|\lambda\|_{E_2(0,T;-A^*)}
        \end{equation*}
        for $\lambda\in E_2(0,T;-A^*)$ and some constants $C,C' > 0$. Thus, the graph norm $\|E_{0,-A^*}\lambda\|_{\dom{\calA_E'}}$ and  $\|\lambda\|_{E_2(0,T;-A^*)}$ are equivalent. Hence, $E_{0,-A^*}$ is an isomorphism onto $\dom{\calA_E'}$ endowed with the graph norm, and the last assertion follows from Proposition~\ref{prop:dual_companion} and Proposition~\ref{prop:companion_unique}.
	\end{proof}
	\begin{rmrk}
		Theorem~\ref{thm:cauchy_companion} also follows from Proposition~\ref{prop:cauchy_companion_identity}, Lemma~\ref{cor:backward_iso} and Proposition~\ref{prop:companion_unique}, which shows that, in our framework, the identification of the dual companion merely requires an integration by parts formula and the bounded invertibility of a candidate backward operator. The direct proof above, however, shows explicitly how the multiplier inherits continuity in time from the backward evolution with adjoint generator $A^*$.
	\end{rmrk}
	For the embedding $E=E_{T,A}$, the identification of the backward evolution space in the sense of Proposition~\ref{prop:companion_unique} is thus given by
	\begin{equation*}
		R=E_{0,-A^*}:E_2(0,T;-A^*)\to\Xd,
	\end{equation*}
	and the corresponding operator $\Ad$ is precisely the backward operator introduced in \eqref{eq:bkw_cauchy}. Conversely, by Corollary~\ref{cor:companion_duality}, the dual companion of $(E_2(0,T;-A^*),\allowbreak E_{0,-A^*},\allowbreak\Ad)$ is represented by $(E_2(0,T;A),\allowbreak E_{T,A},\allowbreak\calA)$, i.e., the forward problem is recovered as the dual companion of the backward problem.

	\smallskip

	\noindent\textbf{The optimal control problem and optimality conditions.} We consider \eqref{eq:OCP} with $\calX=E_2(0,T;A)$, $\calU=L^2(0,T;U)$, $\calB u=(Bu,0)$, $b=(0,x_0)$, $\theta=\iota_{\Uad}$ with $\Uad=\{u\in L^2(0,T;U):u(t)\in U_\mathrm{ad}\text{ for a.e.\ }t\in(0,T)\}$ for a nonempty, closed and convex set $U_\mathrm{ad}\subset U$, and $J(x)=f(E_{T,A}x)$ with
	\begin{equation*}
		f(h,h_T)=\tfrac12\|h-x_d\|^2_{L^2(0,T;X)}+\varphi(h_T)
	\end{equation*}
	for some $x_d\in L^2(0,T;X)$ and a convex and Fr\'echet differentiable function $\varphi:X\to\R$. By Theorem~\ref{thm:inv_of_A}, the state equation $\calA x-\calB u=(0,x_0)$ is equivalent to \eqref{eq:abstract_Cauchy} in the weak sense.
	\begin{crllr}\label{cor:forward_backward}
		The triple $(\xopt,\uopt,\popt)$ satisfies \eqref{eq:optsys} if and only if $\popt=(\lopt,\lopt(0))$ for some $\lopt\in E_2(0,T;-A^*)$ with
		\begin{equation}\label{eq:optsys_semilinear}
			\left\{
			\begin{aligned}
				(\xopt)'&=A_{-1}\xopt+B\uopt, && \xopt(0)=x_0,\\
				-\lopt'-(A^*)_{-1}\lopt&=-(\xopt-x_d), && \lopt(T)=-\nabla\varphi(\xopt(T)),\\
				\uopt&=P_{U_\mathrm{ad}}\big(\tfrac{1}{\alpha}B^*\lopt\big) && \text{a.e.\ in }(0,T),
			\end{aligned}
			\right.
		\end{equation}
		i.e., $(\xopt,\lopt)\in C([0,T];X)^2$ is the mild solution of the forward-backward system
		\begin{align*}
			\xopt(t)&=\mathbb{T}(t)x_0+\int_0^t\mathbb{T}(t-s)B\uopt(s)\,\mathrm{d}s,\\
			\lopt(t)&=-\mathbb{T}^*(T-t)\nabla\varphi(\xopt(T))-\int_t^T\mathbb{T}^*(s-t)\big(\xopt(s)-x_d(s)\big)\,\mathrm{d}s,
		\end{align*}
		for $t\in[0,T]$. In particular, \eqref{eq:optsys_semilinear} has a unique solution.
	\end{crllr}
	\begin{proof}
		By Lemma~\ref{cor:ETA}, $E_{T,A}$ is continuous and injective with dense range, and $f$ is convex and Fr\'echet differentiable with $\nabla f(h,h_T)=(h-x_d,\nabla\varphi(h_T))$. Hence, Theorems~\ref{thm:existence_optimal_sol} and~\ref{thm:adjoint_regularity_general} apply, where, by Theorem~\ref{thm:cauchy_companion}, the dual companion is represented by $(E_2(0,T;-A^*),E_{0,-A^*},\Ad)$. Since $\prox_{\theta/\alpha}=P_{\Uad}$ and $\Uad$ is defined pointwise, we have $(P_{\Uad}v)(t)=P_{U_\mathrm{ad}}(v(t))$ for a.e.\ $t\in(0,T)$. Together with $\calB'(\lambda,\lambda_0)=B^*\lambda$, Theorem~\ref{thm:adjoint_regularity_general} yields the characterization by \eqref{eq:optsys_semilinear}, which is equivalent to the mild formulation by Theorem~\ref{thm:inv_of_A} and Lemma~\ref{cor:backward_iso}. Moreover, the optimality conditions in the proof of Theorem~\ref{thm:existence_optimal_sol} are necessary and sufficient, such that \eqref{eq:optsys} and hence \eqref{eq:optsys_semilinear} have exactly one solution.
	\end{proof}
	In this form, the regularity of state and adjoint is balanced, as they belong to solution spaces of the same type, i.e., $\xopt\in E_2(0,T;A)$ and $\lopt\in E_2(0,T;-A^*)$. In contrast to Subsection~\ref{subsec:parabolic}, no coercivity or analyticity is required, such that, e.g., hyperbolic problems are covered. Moreover, for $X=\R^n$, we have $X_{-1}=X$ and $E_2(0,T;A)=H^1(0,T;\R^n)$ with equivalent norms, such that Corollary~\ref{cor:forward_backward} recovers the forward-backward system \eqref{eq:ODE_fb} of Section~\ref{sec:ode_example}.
	\begin{rmrk}
	In \cite{farkas2025dissipativity}, the optimality system of linear-quadratic optimal control problems governed by $C_0$-groups is decomposed in time by means of an unbounded operator, which is first defined on functions in $H^1(\ell,r;X)\cap L^2(\ell,r;\dom{A})$, called weak solutions in \cite{farkas2025dissipativity}, and then closed. By the above results, the domain of this closure is explicitly given by $E_{r,A}(E_2(\ell,r;A))\times E_{\ell,-A^*}(E_2(\ell,r;-A^*))\times X^2$ on each subinterval $[\ell,r]$. Indeed, up to a bounded perturbation, the operator is given by $\begin{bsmallmatrix}0&-\calA_E'\\ \calA_E&0\end{bsmallmatrix}$ with $E=E_{r,A}$, whose closedness follows from Lemma~\ref{lem:schiela_38_equiv} and Theorem~\ref{thm:cauchy_companion}, and these functions are dense in $E_2(\ell,r;A)$ and $E_2(\ell,r;-A^*)$, see the proof of Proposition~\ref{prop:Green}.
\end{rmrk}
	\subsection{Wave equation with accumulated-state observation}\label{subsec:wave}
	In this subsection, we consider the wave equation. Formally, it is contained in the setting of Subsection~\ref{subsec:cauchy}, but the cost functional is different. In particular, it depends on the accumulated wave state, which leads to an embedding $E$ different from the trace embedding $E_{T,A}$.
	Let $A_0:H\supset\dom{A_0}\to H$ be self-adjoint and uniformly positive on a Hilbert space $H$, let $H_{1/2}\coloneqq\dom{A_0^{1/2}}$ with $\langle\cdot,\cdot\rangle_{H_{1/2}}\coloneqq\langle A_0^{1/2}\cdot,A_0^{1/2}\cdot\rangle_H$, $H_1\coloneqq\dom{A_0}$, and let $H_{-1/2}$ be the dual space of $H_{1/2}$ with respect to the pivot space $H$. Define the \textit{energy space}
\begin{equation*}
    X \coloneqq H_{1/2} \times H, \qquad \left\langle \stack{u}{v}, \stack{z}{w} \right\rangle_X \coloneqq \langle A_0^{1/2} u,A_0^{1/2}z \rangle_{H} + \langle v, w \rangle_H
\end{equation*} and the wave operator
	\begin{equation*}
		A\coloneqq\begin{bmatrix}0&I\\-A_0&0\end{bmatrix},\qquad \dom{A}=H_1\times H_{1/2}.
	\end{equation*}
	As is well known, $A$ is skew-adjoint and thus generates a unitary group on $X$. Moreover, $X_1=X_1^d=H_1\times H_{1/2}$ and $X_{-1}=X_{-1}^d=H\times H_{-1/2}$, and $A_{-1}=-(A^*)_{-1}$ is given by the same matrix with $A_0$ extended to an operator in $L(H_{1/2},H_{-1/2})$. By definition
\begin{align*}
    E_2(0,T;A) &= \left\{ \stack{y}{v} \in C([0,T];H_{1/2} \times H) : y' - v \in L^2(0,T;H_{1/2}) , v' + A_0 y \in L^2(0,T;H) \right\} \\
   \left|\left|\stack{y \vphantom{y'}}{v \vphantom{v'+A_0}} \right|\right|_{E_2(0,T;A)} &= \left|\left|\stack{y \vphantom{y'}}{v \vphantom{v'+A_0}} \right|\right|_{C([0,T];H_{1/2} \times H)} +\left|\left|\stack{y'-v}{v'+A_0 y} \right|\right|_{L^2(0,T;H_{1/2} \times H)} .
\end{align*}
and, in particular, $E_2(0,T;-A^*)=E_2(0,T;A)$.

	\smallskip

	\noindent\textbf{The optimal control problem.} Let $V:L^2(0,T;X)\to L^2(0,T;X)$, $(Vx)(t)\coloneqq\int_0^tx(s)\,\mathrm{d}s$, be the Volterra operator and consider
	\begin{equation*}
		E:E_2(0,T;A)\to\calP,\qquad Ex\coloneqq\stack{Vx}{x(T)}.
	\end{equation*}
	Then $E$ is continuous and injective, as $Vx=0$ implies $x=0$. Moreover, $E$ has dense range. Indeed, for $(g,\xi)\in\calP$, let $g_n\in C_c^2((0,T);X_1)$, $\xi_n\in X_1$ and $\eta_n$ be as in the proof of Lemma~\ref{cor:ETA}. Then $x_n\coloneqq g_n'+\eta_n\xi_n\in C^1([0,T];X_1)$ satisfies $x_n(T)=\xi_n$ and $Vx_n=g_n+V(\eta_n\xi_n)$ with $\|V(\eta_n\xi_n)\|_{L^2(0,T;X)}\leq\tfrac{T^{1/2}}{n}\|\xi_n\|_X$, such that $Ex_n\to(g,\xi)$. For $x_d\in L^2(0,T;X)$, a convex and Fr\'echet differentiable $\varphi:X\to\R$, $\calU=L^2(0,T;H)$ and $\alpha>0$, we consider
	\begin{equation}\label{eq:OCP_wave_accumulated}
		\min_{(x,u)\in E_2(0,T;A)\times L^2(0,T;H)}\ \tfrac{1}{2}\|Vx-x_d\|_{L^2(0,T;X)}^2+\varphi(x(T))+\tfrac{\alpha}{2}\|u\|_{L^2(0,T;H)}^2\quad\mathrm{s.t.}\quad\calA x-\calB u=\stack{0}{x_0},
	\end{equation}
	where $\calB u\coloneqq((0,u),0)$, i.e., the control acts on the second component of the wave equation. This is \eqref{eq:OCP} without control constraints, i.e., with $\theta=0$. Then $J=f\circ E$ with $f(h,h_T)=\tfrac12\|h-x_d\|^2_{L^2(0,T;X)}+\varphi(h_T)$. In contrast to a standard tracking-type cost, the value $(Vx)(t)$ depends on the complete history of the state on $[0,t]$.
	\smallskip

	\noindent\textbf{Dual companion and balanced optimality system.} The dual companion is identified by Proposition~\ref{prop:companion_unique} as follows.
	\begin{prpstn}\label{prop:wave_accumulated_companion}
		The dual companion of $(E_2(0,T;A),E,\calA)$ is represented by
		\begin{equation*}
			\Xd=\left\{\lambda\in E_2(0,T;A):\lambda'-A_{-1}\lambda\in H^1(0,T;X),\ (\lambda'-A_{-1}\lambda)(T)=0\right\},
		\end{equation*}
		together with $\Ed \lambda=(\lambda,\lambda(0))$ and $\Ad \lambda=((\lambda'-A_{-1}\lambda)',\lambda(T))$.
	\end{prpstn}
		\begin{proof}
		By Fubini's theorem, as
		\begin{equation*}
			\langle Vx,z\rangle_{L^2(0,T;X)}=\int_0^T\Big\langle\int_0^tx(s)\,\mathrm{d}s,z(t)\Big\rangle_X\,\mathrm{d}t=\int_0^T\Big\langle x(s),\int_s^Tz(t)\,\mathrm{d}t\Big\rangle_X\,\mathrm{d}s
		\end{equation*}
		for all $x,z\in L^2(0,T;X)$. Thus, the adjoint of the Volterra operator $V$ with respect to the inner product of $L^2(0,T;X)$ is given by
		\begin{equation*}
			(V^*z)(t)=\int_t^Tz(s)\,\mathrm{d}s,\qquad t\in[0,T],
		\end{equation*}
		Moreover, by the fundamental theorem of calculus for Bochner integrals, for $z\in L^2(0,T;X)$, the function $q=V^*z$ belongs to $H^1(0,T;X)$ with $q'=-z$ and $q(T)=0$. Conversely, every $q\in H^1(0,T;X)$ with $q(T)=0$ satisfies
		\begin{equation*}
			q(t)=q(T)-\int_t^Tq'(s)\,\mathrm{d}s=\big(V^*(-q')\big)(t),\qquad t\in[0,T].
		\end{equation*} Hence, \begin{equation*}
			\ran{V^*}=\left\{q\in H^1(0,T;X):q(T)=0\right\}.
		\end{equation*}
		In particular, $V^*$ is injective and $(V^*)^{-1}q=-q'$ for $q\in\ran{V^*}$.

		Since $E=\diag(V,I)E_{T,A}$, we have $E'=E_{T,A}'\diag(V^*,I)$ and hence $\ran{E'}\subset\ran{E_{T,A}'}$. Thus, by Lemma~\ref{lem:schiela_38_equiv} and Theorem~\ref{thm:cauchy_companion}, every $p\in\calP$ with $\calA'p\in\ran{E'}$ is of the form $p=(\lambda,\lambda(0))$ with $\lambda\in E_2(0,T;-A^*)$. As $A$ is skew-adjoint, we have $(A^*)_{-1}=-A_{-1}$ and hence $E_2(0,T;-A^*)=E_2(0,T;A)$. For $\lambda\in E_2(0,T;A)$, the Green identity \eqref{eq:int_by_parts} yields
		\begin{equation*}
			\left\langle\calA h,\stack{\lambda}{\lambda(0)}\right\rangle_\calP=\int_0^T\big\langle-\lambda'(t)+A_{-1}\lambda(t),h(t)\big\rangle_X\,\mathrm{d}t+\langle \lambda(T),h(T)\rangle_X,\qquad h\in E_2(0,T;A),
		\end{equation*}
		i.e., $\calA'(\lambda,\lambda(0))=E_{T,A}'(-\lambda'+A_{-1}\lambda,\lambda(T))$, cf.\ Proposition~\ref{prop:cauchy_companion_identity}. Consequently, $\calA'(\lambda,\lambda(0))\in\ran{E'}$ if and only if there exists $(\eta,\eta_T)\in\calP$ with
		\begin{equation*}
			E_{T,A}'\stack{-\lambda'+A_{-1}\lambda}{\lambda(T)}=E_{T,A}'\stack{V^*\eta}{\eta_T},
		\end{equation*}
		which, by the injectivity of $E_{T,A}'$, see Lemma~\ref{cor:ETA}, holds if and only if $-\lambda'+A_{-1}\lambda=V^*\eta$ and $\lambda(T)=\eta_T$. Hence, $\calA'(\lambda,\lambda(0))\in\ran{E'}$ if and only if $-\lambda'+A_{-1}\lambda\in\ran{V^*}$, which, by the characterization of $\ran{V^*}$, is equivalent to
		\begin{equation*}
			\lambda'-A_{-1}\lambda\in H^1(0,T;X),\qquad \left(\lambda'-A_{-1}\lambda\right)(T)=0,
		\end{equation*}
		i.e., to $\lambda\in\Xd$. For such $\lambda$, Lemma~\ref{lem:schiela_38_equiv} yields
		\begin{equation*}
			\calA_E'\stack{\lambda}{\lambda(0)}=(E')^{-1}\calA'\stack{\lambda}{\lambda(0)}=\begin{bmatrix}(V^*)^{-1}&0\\0&I\end{bmatrix}(E_{T,A}')^{-1}\calA'\stack{\lambda}{\lambda(0)}
		\end{equation*}
		and hence
		\begin{equation*}
			\calA_E'\stack{\lambda}{\lambda(0)}=\stack{(V^*)^{-1}(-\lambda'+A_{-1}\lambda)}{\lambda(T)}=\stack{(\lambda'-A_{-1}\lambda)'}{\lambda(T)}=\Ad \lambda.
		\end{equation*}
		Thus, $\Ed:\Xd\to\dom{\calA_E'}$ is bijective with $\calA_E'\Ed=\Ad$, i.e., $\Ed$ is an isomorphism onto $\dom{\calA_E'}$ endowed with the graph norm if $\Xd$ is endowed with the norm $\|(\lambda,\lambda(0))\|_\calP+\|\Ad \lambda\|_\calP$, and the assertion follows from Proposition~\ref{prop:companion_unique}.
	\end{proof}
	Hence, setting $r\coloneqq \lambda'-A_{-1}\lambda$, the equation $\Ad \lambda=(g,\lambda_T)$ is equivalent to the cascade of the backward problems $r'=g$, $r(T)=0$, and $\lambda'-A_{-1}\lambda=r$, $\lambda(T)=\lambda_T$.
    Consequently, Theorem~\ref{thm:adjoint_regularity_general} yields the following.
	\begin{crllr}\label{cor:wave_optsys}
		The triple $(\xopt,\uopt,\popt)$ satisfies \eqref{eq:optsys} if and only if $\popt=(\lopt,\lopt(0))$ for some $\lopt=(\lopt_1,\lopt_2)\in\Xd$ with
		\begin{equation}\label{eq:optsys_wave}
			\left\{
			\begin{aligned}
				(\xopt)'-A_{-1}\xopt&=\stack{0}{\uopt}, && \xopt(0)=x_0,\\
				\ropt'&=x_d-V\xopt, && \ropt(T)=0,\\
				\lopt'-A_{-1}\lopt&=\ropt, && \lopt(T)=-\nabla\varphi(\xopt(T)),\\
				\uopt&=\tfrac1\alpha\lopt_2,
			\end{aligned}
			\right.
		\end{equation}
		where $\ropt\coloneqq\lopt'-A_{-1}\lopt$ and $\nabla\varphi$ denotes the Riesz representative with respect to the energy inner product. In particular, \eqref{eq:optsys_wave} has a unique solution.
	\end{crllr}
	\begin{proof}
		As $\theta=0$ and hence $\prox_{\theta/\alpha}=I$, $\calB'(q,q_0)=B^*q$ for $Bu=(0,u)$ with $B^*(z,w)=w$ with respect to the energy inner product, and $\nabla f(h,h_T)=(h-x_d,\nabla\varphi(h_T))$, the characterization follows from Theorem~\ref{thm:adjoint_regularity_general} and Proposition~\ref{prop:wave_accumulated_companion} with $\ropt\coloneqq\lopt'-A_{-1}\lopt$, where $\ropt(T)=0$ holds as $\lopt\in\Xd$.
	\end{proof}

\section{Conclusion}\label{sec:conclusion}
We have presented a framework which balances the regularity of state and multiplier in optimality systems of optimal control problems governed by operator equations. The key ingredient is the dual companion, which is unique up to isomorphism and can be identified by an integration by parts formula together with the bounded invertibility of a candidate backward operator. For evolution equations governed by $C_0$-semigroups, this yields a forward-backward optimality system in which state and adjoint belong to solution spaces of the same type, without any assumption on analyticity or maximal regularity.

\section*{Declaration on the use of generative AI}
During the preparation of this work, the authors used Claude Opus 5.5 (Anthropic) for language editing and typesetting. Moreover, the example of a wave equation with a cost functional acting on the accumulated state in Subsection~\ref{subsec:wave} was suggested by this tool. The authors reviewed, verified and edited all content, including all mathematical statements and proofs, and take full responsibility for the content of the publication.

\bibliographystyle{abbrv}
\bibliography{sample}
\end{document}